\documentclass[12pt]{amsart}
\usepackage{multirow} 
\usepackage{amssymb,amsmath, cite, graphicx, url}
\usepackage[curve]{xypic}
\usepackage{caption}
\usepackage{xr}
\makeatletter
\def\@cite#1#2{{\m@th\upshape\bfseries%
[{#1\if@tempswa{\m@th\upshape\mdseries, #2}\fi}]}}
\makeatother
\theoremstyle{plain}
\newtheorem{thm}[subsection]{Theorem}
\newtheorem{cor}[subsection]{Corollary}
\newtheorem{prop}[subsection]{Proposition}
\newtheorem{lem}[subsection]{Lemma}

\theoremstyle{definition}
\newtheorem{rem}[subsection]{Remark}

\numberwithin{equation}{subsection}
\newcommand{\bC}{{\mathbb{C}}}
\newcommand{\bQ}{{\mathbb{Q}}}
\newcommand{\bR}{{\mathbb{R}}}

\newcommand{\bH}{{\mathbb{H}}}
\newcommand{\bZ}{{\mathbb{Z}}}

\newcommand{\bN}{{\mathbb{N}}}

\newcommand{\M}{{\mathcal{M}}}
\newcommand{\A}{{\mathcal{A}}}
\newcommand{\h}{{\mathcal{h}}}

\newcommand{\C}{{\mathcal{C}}}

\newcommand{\T}{{\mathcal{T}}}
\renewcommand{\S}{{\mathcal{S}}}
\newcommand{\col}{\operatorname{col}}

\newcommand{\End}{\operatorname{End}}
\newcommand{\Jac}{\operatorname{Jac}}

\newcommand{\Span}{\operatorname{span}}

\newcommand{\lcm}{\operatorname{lcm}}

\newcommand{\br}{{\mathbf{r}}}

\newcommand{\CP}{{\mathbb C}\!\operatorname{P}^1}

\renewcommand{\h}{\mathfrak{h}}
\newcommand{\R}{\mathcal{R}}
\renewcommand{\P}{\mathcal{P}}

\begin{document}

\title[The Veech-Ward-Bouw-M\"oller curves]{Schwarz triangle mappings and Teichm\"uller curves:\\
the Veech-Ward-Bouw-M\"oller curves}
%
\author[A.Wright]{Alex~Wright}
\address{Math.\ Dept.\\U. Chicago\\
5734 S. University Avenue\\
Chicago, Illinois 60637}
\email{alexmwright@gmail.com}
%
\date{}

\begin{abstract}
We study a family of Teichm\"uller curves $\T(n,m)$ constructed by Bouw and M\"oller, and previously by Veech and Ward in the cases $n=2,3$. We simplify the proof that $\T(n,m)$ is a Teichm\"uller curve, avoiding the use M\"oller's characterization of Teichm\"uller curves in terms of maximally Higgs bundles. Our key tool is a description of the period mapping of $\T(n,m)$ in terms of Schwarz triangle mappings.

We prove that $\T(n,m)$ is always generated by Hooper's lattice surface with semiregular polygon decomposition. We compute Lyapunov exponents, and determine algebraic primitivity in all cases. We show that frequently, every point (Riemann surface) on $\T(n,m)$ covers some point on some distinct $\T(n',m')$. 

The $\T(n,m)$ arise as fiberwise quotients of families of abelian covers of $\CP$ branched over four points. These covers of $\CP$ can be considered as abelian parallelogram-tiled surfaces, and this viewpoint facilitates much of our study.
\end{abstract}

\maketitle

\setcounter{tocdepth}{1}

\tableofcontents
\newpage

%
%
%

\section{Introduction}\label{S:intro}

A \emph{Teichm\"uller curve} is an isometrically immersed curve in the moduli space of genus $g$ curves $\M_g$, with respect to the Teichm\"uller metric. Teichm\"uller curves give rise to billiards and translation surfaces with optimal dynamical properties \cite{V}, and have a rich and interesting algebro-geometric theory \cite{M, M2}. In analogy with lattices in Lie groups they are either arithmetic (generated by square-tiled surfaces) or not, and come in groups analogous to commensurability classes. In the non-arithmetic case each such commensurability class of Teichm\"uller curves contains a unique extremal element which is called \emph{primitive} \cite{M2}. Also in analogy with lattices (say in $PU(n,1)$) non-arithmetic Teichm\"uller curves seem to be quite rare. The only currently known primitive examples in genus greater than four are  the subject of this paper.

\textbf{The Veech-Ward-Bouw-M\"oller curves.} The \emph{flat pillowcase} is obtained by gluing two isometric squares together, giving a flat metric on $\CP$. Its symmetry group is the Klein four group $\bZ_2\times\bZ_2$. 

\begin{figure}[h]
\includegraphics[scale=0.5]{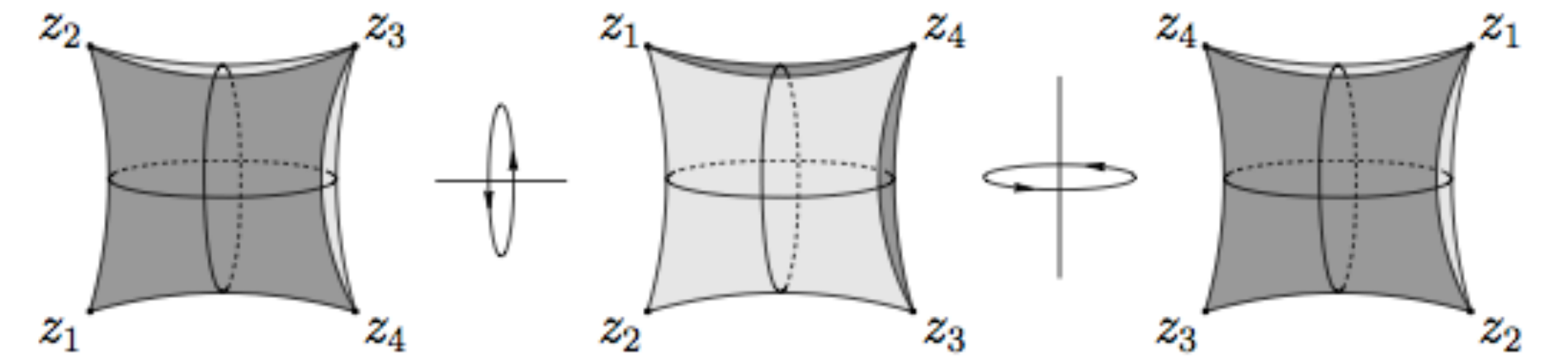}
\caption{The action of two pillowcase symmetries on the flat pillowcase (center).}
\label{F:sigmas}
\end{figure}

More generally, given any four points $z_1, z_2, z_3, z_4\in\CP$, there is a flat metric on $\CP$ obtained by gluing two isometric parallelograms whose corners are the $z_j$. The symmetry group is again the Klein four group, which acts on $\CP$ by M\"obius transformations.     

An \emph{abelian parallelogram-tiled surface} is an abelian cover of $\CP$ branched over at most four points $z_j$ and equipped with a lift of the parallelogram-tiled metric \cite{W1}. In this paper the parallelogram-tiled flat structure is not essential, but it is natural for the flat geometer to keep it in mind. 

Given any cover of $\CP$ branched over $z_1=0, z_2=1, z_3=\lambda, z_4=\infty$ we may vary $\lambda$ to obtain a family of Riemann surfaces over the base $\CP\setminus\{0,1,\infty\}\ni\lambda$. 

For each pair $n,m>1$ with $nm\geq 6$ we will consider a very special family of this type,
$$\pi:\S(n,m)\to \CP\setminus\{0,1,\infty\}. $$ 
The fibers $\pi^{-1}(\lambda)$ are abelian parallelogram-tiled surfaces $S(n,m)$ which are exceptionally symmetric in that they admit a very nice lift of the pillowcase symmetry group $\bZ_2\times \bZ_2$. 

Informally the Veech-Ward-Bouw-M\"oller curve $\T(n,m)$ is the closure in moduli space of the image of the fiberwise quotient map 

$$q:\lambda\mapsto \pi^{-1}(\lambda)/(\bZ_2\times\bZ_2).$$ 

More formally, $\T(n,m)$ is a map from a curve to moduli space (and moreover a base change is required to define $q$). It is the same curve in moduli space that is considered in different language by Bouw-M\"oller, who show that it is a Teichm\"uller curve. Moreover, Bouw-M\"oller show that the $\T(2,m)$ and $\T(3,m)$ are the Teichm\"uller curves considered by Veech and Ward respectively. 

We will give a simplified proof that $\T(n,m)$ is a Teichm\"uller curve. The key ideas are presented in Section 2, but we will hint briefly at the proof before proceeding to describe our new results. 

Royden's Theorem asserts that the Teichm\"uller metric is the same as the Kobayashi metric \cite{Hub, IT}. The magic of the Kobayashi metric is that holomorphic maps are distance nonincreasing, so in particular if the composition of two maps is an isometry, then each of the two maps is also. So to show $\T(n,m)\to\M_g$ is an isometry, it will suffice to show that a single period coefficient on $\T(n,m)$ is an isometry. More precisely, $\T(n,m)$ will be lifted to Torelli space, and a specific entry in the period matrix will shown to be an isometry. 

This program is feasible because the period mapping of $\S(n,m)$ may be completely described in terms of \emph{Schwarz triangle mappings}, which are biholomorphisms from the upper half plane to (in our case hyperbolic) triangles \cite{W1}. 

\textbf{Covering relations.} We have discovered that

\begin{thm}\label{T:cov} Suppose $n'$ divides $n$, and $m'$ divides $m$. If $n$ and $m$ are even suppose also that $n/n'+m/m'$ is even. Then every point (Riemann surface) on $\T(n,m)$ covers a point (Riemann surface) on $\T(n',m')$. 
\end{thm}

For example, every point on the Veech Teichm\"uller curve $\T(2,24)$ generated by the regular $24$--gon covers some point on the Veech Teichm\"uller curve $\T(2,8)$ generated by the regular $8$--gon. In other words, given any translation surface $(X,\omega)$ in the $SL(2,\bR)$--orbit of the regular $24$--gon, there is some translation surface $(X', \omega')$ in the $SL(2,\bR)$--orbit of the regular $8$--gon, so that there is a covering of Riemann surfaces $X\to X'$. 

This is surprising because there is no hint that such a result should be true from the flat geometry. 

%

\textbf{Arithmetic origins.} It is also surprising that \emph{arithmetic} Teichm\"uller curves (generated by the abelian square-tiled surfaces $S(n,m)$) can be used to construct the \emph{non-arithmetic} $\T(n,m)$. A direct consequence of the construction is

\begin{thm}\label{T:arith}
A quadratic differential with simple poles may be assigned to all but finitely many points (Riemann surfaces) on $\T(n,m)$, giving each of these Riemann surfaces the structure of a parallelogram-tiled surface.

$\T(n,m)$ is the closure a Hurwitz curve of covers of $\CP$ branched over four points.
\end{thm}

Here a \emph{Hurwitz curve} is the closure of a space of covers of $\CP$ branched over four points, where all the covers are topologically the same. That is, a Hurwitz curve results from taking a cover of $\CP$ branched over four points and varying the location of the branch points. 

The first part of Theorem \ref{T:arith} is rather surprising: it is unexpected that so many Riemann surfaces (the points of $\T(n,m)$) could support the structure of a lattice surface in two different ways (one square-tiled and one not).

\begin{proof}
The image of $\T(n,m)$ in moduli space is equal to closure of the image of the fiberwise quotient map $q$, by construction. We will show in Proposition \ref{P:smooth} and Theorem \ref{T:BM} that all but finitely many points on $\T(n,m)$ are in the image of the fiberwise quotient map and hence are of the form  $S(n,m)/(\bZ_2\times\bZ_2)$.

The pillowcase symmetry group preserves the parallelogram-tiling, so the quotient of the exceptionally symmetric square-tiled surface $S(n,m)$ by the pillowcase group $\bZ_2\times\bZ_2$ is again parallelogram-tiled. The parallelogram-tiled metric of the quotient has cone angles of $\pi$.
\begin{figure}[h]
\includegraphics[scale=0.6]{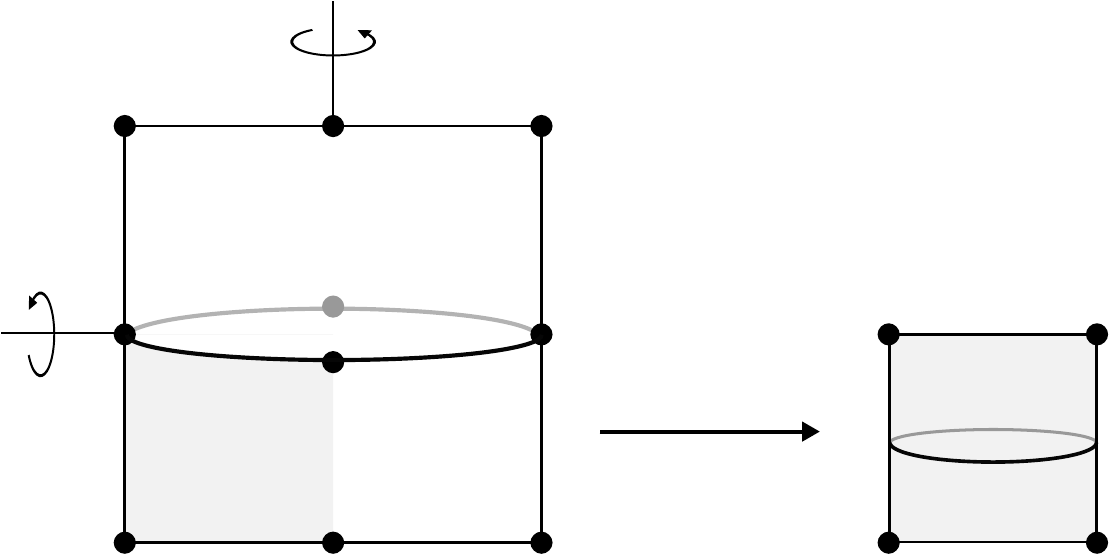}
\caption{The quotient of the flat pillowcase by the pillowcase symmetry group is again a flat pillowcase, of one fourth the area.}
\label{F:4-cover}
\end{figure}

Each point in the image of $q$ admits a cover 
$$S(n,m)/(\bZ_2\times\bZ_2)\to \CP/(\bZ_2\times\bZ_2) \simeq \CP.$$
This map is branched over four points, and it follows easily that $\T(n,m)$ is (up to closure) the space of such covers. Hence $\T(n,m)$ is the closure of a Hurwitz curve.
\end{proof}

\textbf{Real multiplication of Hecke type.} M\"oller has shown that Techm\"uller curves parameterize Riemann surfaces whose Jacobians have a factor with real multiplication (an inclusion of a totally real number field into the endomorphism algebra) \cite{M}. 

If $X$ is a Riemann surface, endomorphisms of $\Jac(X)$ can be considered as ``hidden symmetries" of $X$. Sometimes, they arise from honest symmetries, that is, they are induced by automorphisms of $X$. Ellenberg has studied situations in which $X$ has no automorphisms, but $\Jac(X)$ admits real multiplication which arises from automorphisms of a finite cover of $X$ \cite{E}. In such cases he defines the real multiplication on  $\Jac(X)$ to be of \emph{Hecke type}. See Section \ref{S:Hecke} for definitions. 

\begin{thm}\label{T:Hecke}
The real multiplication on a factor of the Jacobians of all but finitely many points (Riemann surfaces) on $\T(n,m)$ guaranteed by \cite{M} is of Hecke type. That is, the endomorphisms of the Jacobian, which together form real multiplication, come from deck transformations of the exceptionally symmetric square-tiled surfaces which cover points (Riemann surfaces) on $\T(n,m)$. 
\end{thm}

\textbf{Generators.} In Section \ref{S:eqns}, extending work of Bouw-M\"oller in the case when $n$ and $m$ are relatively prime, we compute holomorphic one forms which generate each $\T(n,m)$; the formulas and corollaries are given in Section \ref{S:grn}. 

Hooper has given an elementary construction of Teichm\"uller curves which are generated by translation surfaces with a particularly beautiful flat structure having a \emph{semiregular polygon decomposition} \cite{H}. These flat surfaces were discovered independently by Ronen Mukamel. Comparing our generators to Hooper's, we obtain the next theorem, which again was previously known in the case where $n$ and $m$ are relatively prime. 

\begin{thm}\label{T:same}
The Teichm\"uller curves constructed by Hooper are the same as the Veech-Ward-Bouw-M\"oller curves. 
\end{thm}

Theorem \ref{T:same} answers a question of Hooper \cite[Question 19]{H}. As a corollary of Theorem \ref{T:same}, we list some of Hooper's results, which were obtained by Hooper using the semiregular polygon decomposition. Let $\Delta^+(n,m,l)\subset PSL(2,\bR)$ be the orientation-preserving part of the group generated by reflections in the sides of a hyperbolic triangle with angles $\frac\pi{n}, \frac\pi{m}, \frac{\pi}{l}$.

The uniformizing group of a Teichm\"uller curve is frequently called its \emph{Veech group}. Definitions of primitivity appear below. 

\begin{cor}\label{C:hoop} Let $n,m>1$ with $nm\geq6$, and set $\gamma=\gcd(n,m)$. 
$\T(n,m)=\T(m,n)$. $\T(n,m)$ is arithmetic if and only if $(m,n)$ or $(n,m)$ is in the set $$\{(2,3),(2,4),(2,6),(3,3),(4,4),(6,6)\}.$$ $\T(n,m)$ is geometrically primitive when it is not arithmetic.  

The uniformizing group of $\T(n,m)$ is: 

\begin{tabular}[t]{ll}
$\Delta^+(n,m,\infty)$ & if $n\neq m$ and $m$ or $n$ is odd;\\
an index two subgroup & if $n\neq m$ are both even;\\
$\Delta^+(2,n,\infty)$ & if $n=m$ is odd;\\
$\Delta^+(n/2,\infty,\infty)$ & if $n=m$ is even. \\
\end{tabular}

The curve $\T(n,m)\subset \M_g$, where the genus is 
\begin{equation*}
g = \begin{cases}
\frac{(n-1)(m-1)}{2}+\frac{1-\gamma}{2} & \text{if $n$ or $m$ odd,}\\
\frac{(n-1)(m-1)}{4}+\frac{3-2\gamma}{4} & \text{if $n$ and $m$ even, and $n/\gamma$ and $m/\gamma$ odd,}\\
\frac{(n-1)(m-1)}{4}+\frac{3-\gamma}{4}  & \text{otherwise.}
\end{cases}
\end{equation*}

In all but the last case, $\T(n,m)$ is generated by an abelian differential with $\gamma$ zeros of equal order; in the last case, there are $\gamma/2$.
\end{cor}

Some of these results can also be obtained by other means and were obtained by \cite{BM}, but at least when $n$ and $m$ are not relatively prime the only known proof of primitivity and exact calculation of the uniformizing group are due to Hooper.

\begin{figure}[h]\label{F:Semi-reg}
\includegraphics[scale=0.6]{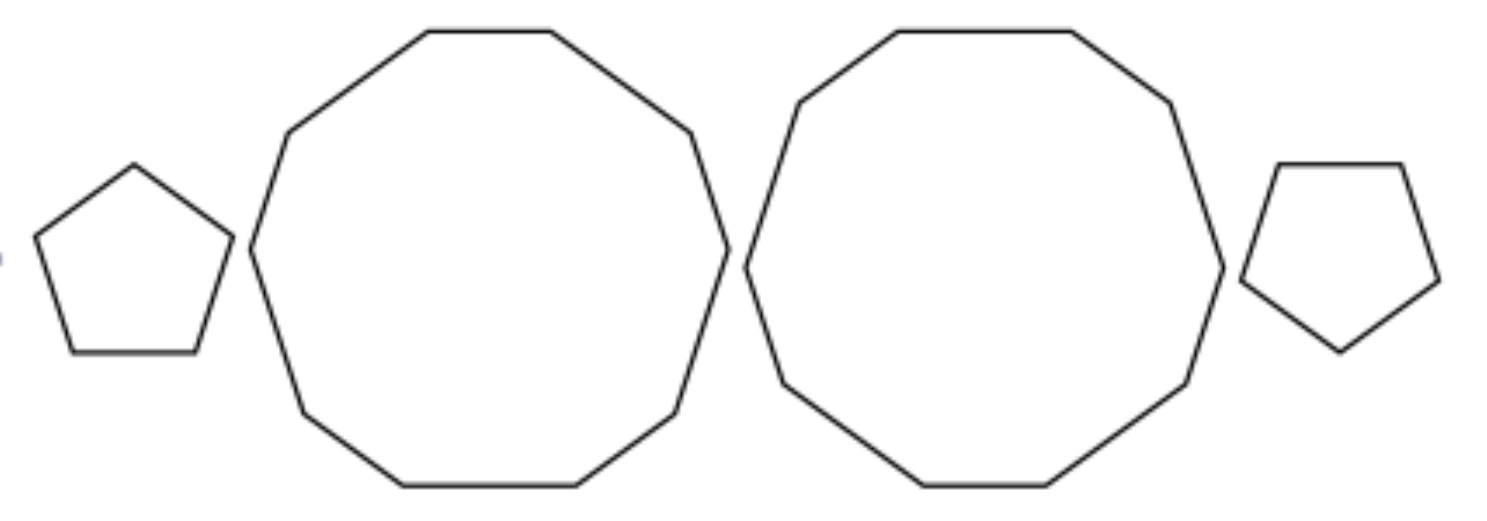}
\caption{An example of the semiregular polygon decomposition discovered by Hooper and Mukamel. Identifying parallel edges we obtain a generator for $\T(4,5)$.}
\end{figure}

\textbf{Lyapunov exponents.} In flat geometry the significance of Lyapunov exponents is twofold: they describe both the
dynamics of the Teichm\"uller geodesic flow on moduli space, and the deviation of ergodic averages for straight line flow on the translation
surface \cite{F}. For background and motivation on Lyapunov exponents in flat geometry, see for example \cite{Fex1,EKZbig}. In parallel to the case of abelian square-tiled surfaces handled in \cite{W1}, at the end of Section \ref{S:BM} we determine all the Lyapunov exponents of $\T(n,m)$, and clarify the relationship to the Schwarz triangles used to describe the period mappings. Previously the Lyapunov exponents were given by Bouw-M\"oller when $n$ and $m$ are odd and relatively prime. A corollary of our computation is 

\begin{cor}\label{C:Lmults}
The Lyapunov spectrum of $\T(n,m)$ consists of nonzero multiples of $\frac{\gcd(n,m)}{nm-n-m}$. In particular, there are never any zero Lyapunov. 
\end{cor}

The Lyapunov exponents however often are not all distinct. See for example the tables in figure \ref{F:tables}. 

\textbf{Primitivity.} A Teichm\"uller curve is called \emph{(geometrically) primitive} if it does not arise from a translation covering construction. A Teichm\"uller curve in $\M_g$ is called \emph{algebraically primitive} when the trace field of the uniformizing  group has degree $g$ over $\bQ$. This is exactly the case when there is real multiplication on the Jacobian, instead of only a factor of the Jacobian. The $\T(n,m)$ are always geometrically primitive but are usually not algebraically primitive.

\begin{thm}\label{T:ap}
Assuming the Teichm\"uller curve $\T(n,m)$ is not arithmetic, it is algebraically primitive if and only if one of $n,m$ is $2$ and the other is a prime, twice a prime, or a power of two. 
\end{thm} 
 
See \cite{E} and \cite{CM} for a summary of the very few known families of curves with real multiplication, and known curves with complex multiplication. Any cone point of an algebraically primitive Teichm\"uller curve has complex multiplication. 
 
%
%
%

\textbf{Notes and references.} There are only very few examples of primitive Teichm\"uller curves known. They are the Prym curves in genus 2, 3 and 4 \cite{Ca, Mc, Mc4}; the Veech-Ward-Bouw-M\"oller curves \cite{BM}; and two sporadic examples, one due to Vorobets in $H(6)$, and another due to Kenyon-Smillie in $H(1,3)$ \cite{HS, KS}. These sporadic examples correspond to billiards in the  $(\pi/5, \pi/3, 7\pi/15)$ and $(2\pi/9, \pi/3, 4\pi/9)$ triangles respectively, and both Teichm\"uller curves are algebraically primitive. 

Teichm\"uller curves generated by abelian differentials are classified in $\M_2$, and there are some finiteness results in higher genus \cite{Bam, M3}, but classification even in $\M_3$ appears difficult.

The Bouw-M\"oller construction is quite novel, and originally used M\"oller's characterization of Teichm\"uller curves involving maximally Higgs bundles. Our proof that $\T(n,m)$ is isometrically immersed is a simplification of theirs; our contributions are to avoid M\"oller's characterization, and to use Schwarz triangle mappings, which are not used in \cite{BM} but allow for a geometric understanding. We also ground the arguments in more elementary language, and point out the connection to the geometry and combinatorics of square-tiled surfaces.

For those results that were previously only known in some cases (for example, $n$ and $m$ relatively prime), no new ideas are required to extend the results to all cases. Our contribution here is to use notation which avoids the case distinctions which pervade \cite{BM}. This being said, often we do not use the methods of \cite{BM} to obtain these results, preferring new approaches of a more geometric flavor.

Veech and Ward gave flat geometry proofs that $\T(n,m)$ is a Teichm\"uller curve in the case $n=2,3$. In light of Theorem \ref{T:same}, Hooper has done the same for all $n$ and $m$. These flat geometry proofs are more elementary than the proof we present, but the proof we present also has a number of advantages. First and foremost, our proof is closer to how Bouw and M\"oller \emph{discovered} $\T(n,m)$ in the first place, and it is gratifying to understand their leap of intuition that cyclic or abelian covers of $\CP$ might be the building blocks for some Teichm\"uller curves uniformized by triangle groups. Second, it allows for the computation of Lyapunov exponents, and an understanding of the period map and monodromy. Third, it allowed us to discover many of the results in this paper.    

We use a number of results from \cite{W1}, where we have developed the theory of abelian square-tiled surfaces. Most readers wishing to understand all the details of this paper will wish to consult this source.  

\textbf{Table of Schwarz triangles.} The reader is encouraged to study figure \ref{F:tables}, where a description of the period map of many $\T(n,m)$ is presented. Many results in this paper are reflected in these tables. 

\textbf{Acknowledgements.}  This research was supported in part by the National Science and Engineering Research Council of Canada, and was partially conducted during the Hausdorff Institute's trimester program ``Geometry and dynamics of Teichm\"uller space."  
The author thanks Alex Eskin, Howard Masur, Martin M\"oller, and Anton Zorich for their instruction and encouragement, and Matt Bainbridge, Irene Bouw, Jordan Ellenberg and Madhav Nori for helpful and interesting discussions. The author is grateful to Anton Zorich and Pat Hooper for allowing us to reproduce figures, and Jennifer Wilson for producing figures.


\section{Key ideas for the study of $\T(n,m)$}\label{S:ideas}

Here we give the main ingredients in the proof that $\T(n,m)$ is a Teichm\"uller curve, and hint at the proof. This section is intended as an extension of the introduction.  

\subsection{Exceptionally symmetry.} In the notation of \cite{W1}, $S(n,m)$ is defined as $$M_{2nm}\left(\begin{array}{cccc}
nm-n-m & nm+n-m & nm+n+m & nm-n+m \\
nm+n-m & nm-n-m & nm-n+m & nm+n+m\\
\end{array}\right).$$

That is, given $w_1, w_2$ in the algebraic closure of $\bC(z)$ such that 
\begin{eqnarray*}
w_1^{2nm}&=&z^{nm-n-m} (z-1)^{nm+n-m} (z-\lambda)^{nm+n+m},\\
w_2^{2nm}&=&z^{nm+n-m} (z-1)^{nm-n-m} (z-\lambda)^{nm-n+m},
\end{eqnarray*}
$S(n,m)$ is defined to be the cover of $\CP$ with function field $\bC(z)[w_1,w_2]$. The base $\CP$ has function field $\bC(z)$. The dependance on $\lambda$ is suppressed in this notation. The natural flat structure on $\CP$ with singularities at $0,1,\lambda, \infty$ may be lifted to $S(n,m)$, giving it the structure of a parallelogram-tiled surface.

We begin with the exceptional symmetry of $S(n,m)$, which is visible in the flat geometry. Denote by $\sigma_j$ the pillowcase symmetry that sends $z_1$ to $z_j$. Covering space theory guarantees that each involution $\sigma_j$ can be lifted to an involution $\tilde{\sigma_j}$ of $S(n,m)$. However, we will require \emph{commuting lifts} (a lift of pillowcase symmetry \emph{group} $\bZ_2\times \bZ_2$), which moreover have special properties.

\begin{figure}[h]\label{F:FlatPillow}
\includegraphics[scale=0.17]{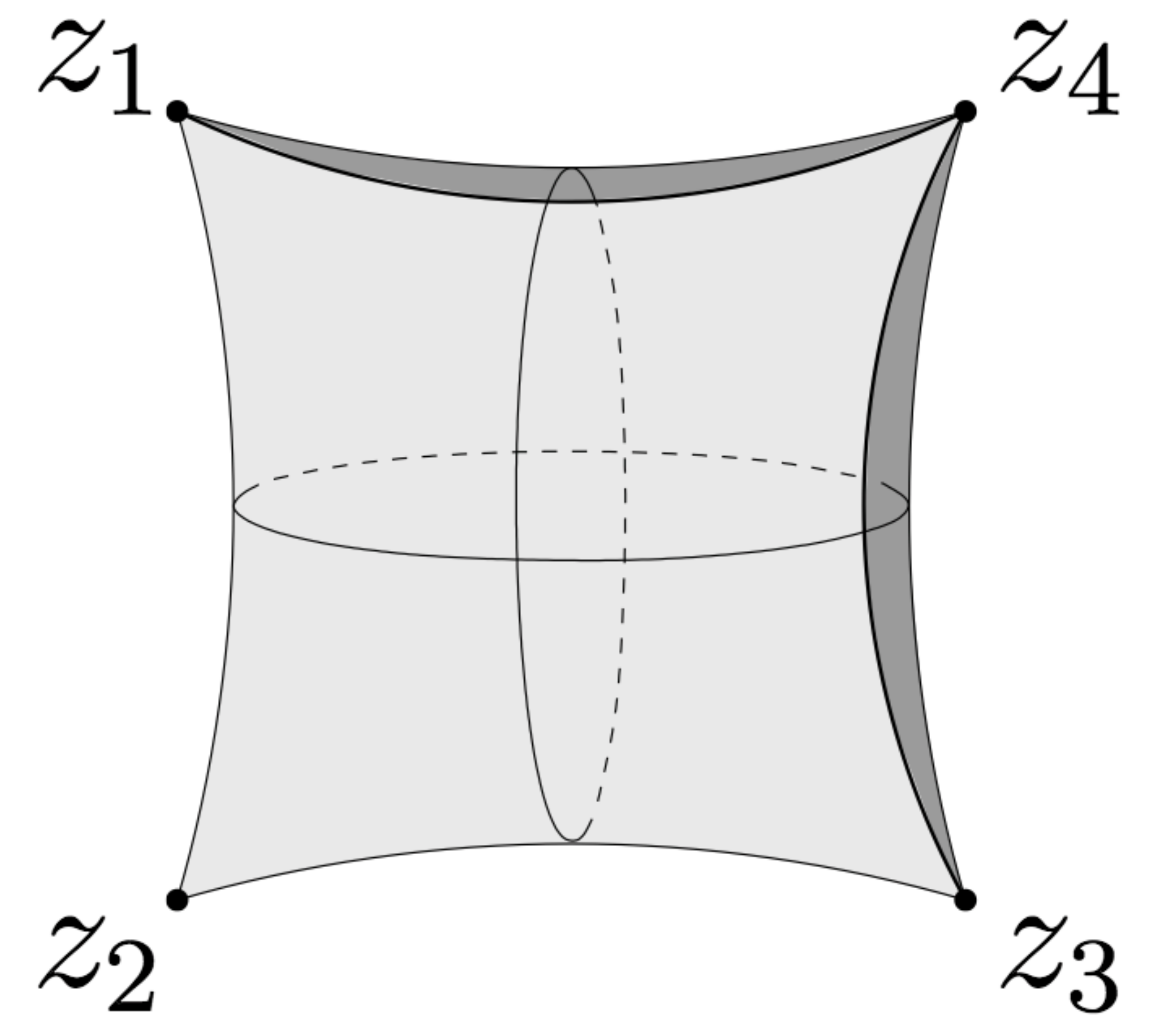}
\caption{The flat pillowcase.}
\end{figure}

Let $\alpha_i$ be the oriented loop about $z_i$. In the standard square-tiled metric, $\alpha_4\circ\alpha_1$ is the core curve of the horizontal cylinder, and $\alpha_2\circ \alpha_1$ is the core curve of the vertical cylinder. So, ``lifts" of $\alpha_4\circ \alpha_1$ to the square-tiled $S(n,m)$ are the core curves of the horizontal cylinders on $S(n,m)$. (We will also use parallelogram-tiled $S(n,m)$, where this language does not apply. By a ``lift" of $\alpha_4\circ \alpha_1$ we mean any unoriented simple closed curve that projects to a multiple of $\alpha_4\circ \alpha_1$.)  

\begin{prop}\label{P:kleinlift}
Up to the action of the deck group, $S(n,m)$ has a unique lift $\langle \tilde{\sigma_2},  \tilde{\sigma_4}\rangle $ of the pillowcase symmetry group so that both $\tilde{\sigma_2}$ and  $\tilde{\sigma_4}$ each have at least one fixed point, unless $n$ and $m$ are both even, in which case there are two such lifts.

In any such lift, the involution $\tilde{\sigma_2}$  maps each lift of the unoriented curve $\alpha_4\circ\alpha_1$ to itself, and the involution $\tilde{\sigma_4}$ maps each lift of $\alpha_2\circ \alpha_1$ to itself.
\end{prop}

It is not at all obvious that Proposition \ref{P:kleinlift} should be true. The proof is straightforward and unenlightening, and  is deferred to Section \ref{S:Snm}. 

\subsection{Schwarz triangle mappings.} Recall that $S(n,m)$ is branched over $z_1=0, z_2=1, z_3=\lambda, z_4=\infty$. By varying $\lambda$, we obtain a family $\pi:\S(n,m)\to B_0$ over the base $B_0=\CP\setminus\{0,1,\infty\}\ni \lambda$. The families $\S(n,m)$ are chosen so that the following result holds. By the \emph{row span} of $S(n,m)$, we mean the abelian subgroup of $(\bZ/(2nm\bZ))^4$ generated by the rows of the matrix
\begin{eqnarray*}
\left(\begin{array}{cccc}
nm-n-m & nm+n-m & nm+n+m & nm-n+m \\
nm+n-m & nm-n-m & nm-n+m & nm+n+m\\
\end{array}\right),
\end{eqnarray*}
whose entries are modulo $2nm$. (In \cite{W1}, we saw that this row span is in bijection to a basis of the function field of $S(n,m)$ over $\bC(z)$, which is why its use is pervasive). For $$\br=(r_1, r_2, r_3, r_4)$$ in the row span of $A$, we defined $$t_j(\br)=\{\frac{r_j}{2nm}\} \quad\text{and}\quad t(\br)=\sum_{j=1}^4 t_j(\br),$$ where $\{\cdot\}$ denotes fractional part.

\begin{prop}\label{P:periods}
Consider the bundle $H^1$ over $B_0$ whose fiber over $\lambda$ is $H^1(\pi^{-1}(\lambda), \bC)$. There exists a direct sum decomposition of this bundle $$H^1=\bigoplus_\br H^1(\br)$$ with the following properties.  The summation is over $\br$ in the row span of $S(n,m)$. The subbundle $H^1(\br)$ is nonzero if and only if $t(\br)=2=t(-\br)$.

Each nonzero $H^1(\br)$ is a flat rank two subbundle whose $(1,0)$ and $(0,1)$ parts each have dimension $1$.  Set  $t_i=t_i(-\br)$  and $$\kappa=|1-t_1-t_3|=0,\quad
\mu=|1-t_2-t_3|\in \frac{1}{m}\bZ,\quad
\nu=|1-t_1-t_2|\in \frac{1}{n} \bZ.$$

The $(1,0)$ part of each $H^1(\br)$ admits a global section $\omega_\lambda$, and homology classes $\alpha, \beta$ may be chosen so that the period mapping 
$$p(\lambda)=\frac{\int_\alpha \omega_\lambda}{\int_\beta \omega_\lambda}$$
is a Schwarz triangle mapping which maps $\bH\subset B_0$ to a hyperbolic triangle with angles $\kappa\pi, \mu\pi, \nu\pi$ at $p(0), p(1), p(\infty)$ respectively. 
\end{prop}  

This follows directly from \cite[Propositions \ref{P:dim} and \ref{P:pmap}]{W1}. We always use (co)homology with coefficients in $\bC$ (not $\bR$). 

The relationship between Propositions \ref{P:kleinlift} and \ref{P:periods} is given by the following lemma, whose proof is also deferred to Section \ref{S:Snm} following the proof of Proposition \ref{P:kleinlift}.

\begin{lem}\label{L:action}
The Klein group action on the $H^1(\br)$ is given by
$$\tilde{\sigma_2} H^1(r_1, r_2, r_3, r_4)= H^1(r_2, r_1, r_4, r_3),$$  $$\tilde{\sigma_4} H^1(r_1, r_2, r_3, r_4) = H^1(r_4, r_3, r_2, r_1).$$
When $\tilde{\sigma}H^1(\br)=H^1(\br)$ for some involution $\tilde{\sigma} \in \bZ_2\times \bZ_2$, then some non-trivial involution in $\bZ_2\times \bZ_2$ acts by negation on $H^1(\br)$. 
\end{lem}

\subsection{The fiberwise quotient map.} Each fiber $\pi^{-1}(\lambda)$ is an $S(n,m)$ and admits a lift of $\bZ_2\times\bZ_2$ by Proposition \ref{P:kleinlift}. After a base change $B\to B_0$ (see Section \ref{SS:BC} for details), a continuous choice of  $\bZ_2\times\bZ_2$ is possible, and we achieve an action of $\bZ_2\times \bZ_2$ on the entire family of covers of $\CP$.

By Lemma \ref{L:action}, after taking fiberwise quotients, the bundle $H^1$ is still the sum of rank two bundles $H^1(\br)$, each of whose period map is a still a Schwarz triangle mapping with the same angles, one each for each group of four distinct $H^1(\br)$ over $S(n,m)$ which are permuted by the Klein four group. (See Lemma \ref{L:decomp} for details.)

Note that we use the notation $H^1(\br)$ to denote both a subbundle of the cohomology of fibers of $\S(n,m)$ and also $\T(n,m)$. We hope this will not be too misleading, despite the fact that each $H^1(\br)$ for $\T(n,m)$ corresponds to a group of four isomorphic $H^1(\br)$ for $\S(n,m)$ which are permuted by the $\bZ_2\times \bZ_2$ action. The notation is justified because as a bundle (rank two complex VHS) over $B$,  $H^1(\br)$ always denotes the same object. 

If $\br$ is the first row of $S(n,m)$,
$$\br= (nm-n-m , nm+n-m , nm+n+m , nm-n+m),$$
 in Proposition  \ref{P:periods} we can calculate $(\kappa , \mu, \nu)=(0, \frac\pi{m}, \frac\pi{n})$. Lemma \ref{L:action} gives that the four bundles $\tilde{\sigma}H^1(\br)$ are distinct. They yield a single rank 2 bundle $H^1(\br)$ after fiberwise quotient, whose period mapping is again described via Schwarz mapping onto a triangle with angles $0, \frac\pi{m}, \frac\pi{n}$. The period map $p_1$ of $H^1(\br)$ will be shown to be an isometry. 
 
\subsection{Removable singularities.} We must of course use the orbifold structure on $\T(n,m)$ induced from the orbifold structure of $\M_g$, and assign to $\T(n,m)$ the unique hyperbolic metric guaranteed by uniformization. The main subtlety is that this does \emph{not} correspond to the hyperbolic metric on $B_0=\CP\setminus\{0,1,\infty\}$. Indeed, the fiberwise quotient map has removable singularities; it sends some points at infinity to interior points, which turn out to be orbifold points. We remind the reader that if some of the punctures on $B_0$ are filled in, many possible orbifold hyperbolic metrics might result, depending on the cone angle assigned to the points which have been filled in. In particular, the hyperbolic metrics $\bH/\Delta^+(n,m,\infty)$ appear. 

To see why the fiberwise quotient map has removable singularities, consider for a moment the flat pillowcase.  The horizontal core curve may be pinched, giving a noded Riemann surface with two genus zero components. This is a degeneration (point at infinity) of $\M_{0,4}$. One of the involutions in the pillowcase symmetric group preserves the pinched curve and interchanges the two components, so the quotient is a smooth genus zero surface.

\begin{figure}[h]
\includegraphics[scale=0.5]{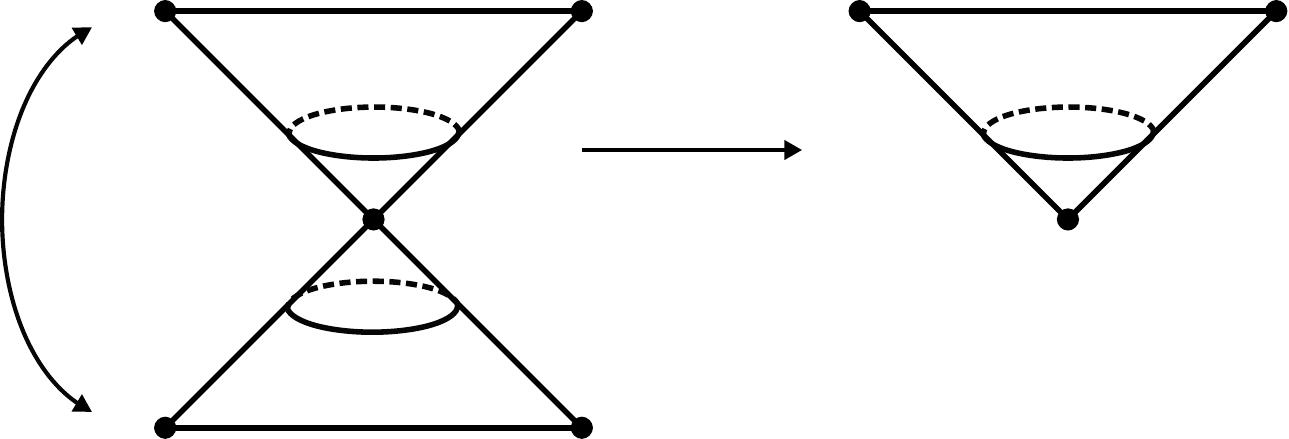}
\caption{The involution $\sigma_2$ extends to the flat pillowcase with the horizontal curve pinched. The quotient is smooth.}
\label{F:Pinched}
\end{figure}

Proposition \ref{P:kleinlift} allows a similar discussion for $\S(n,m)$. 

\begin{prop}\label{P:smooth}
As $\lambda\to 1$, in the Deligne-Mumford compactification $\pi^{-1}(\lambda)$ converges to the noded Riemann surface resulting from pinching the core curves of all horizontal cylinders on the square-tiled surface $S(n,m)$. The quotient of this noded Riemann surface by $\tilde{\sigma_2}$ is smooth. 

Similarly as $\lambda\to \infty$, $\pi^{-1}(\lambda)$ converges to the noded Riemann surface resulting from pinching the core curves of all vertical cylinders on the square-tiled surface $S(n,m)$. The quotient of this noded Riemann surface by $\tilde{\sigma_4}$ is smooth.
\end{prop} 

To be more precise, we should say as $\lambda$ goes to any lift in $B$ of the puncture at $1$ in $B_0$, instead of saying $\lambda\to 1$. 

\begin{proof}
As $\lambda\to 1$, in $\M_{0,4}$ the base converges to two $\CP$'s glued together at a node. One contains the first and fourth marked points ($z_1$ and $z_4$) and the other contains the second and third. This noded Riemann surface is the result of pinching the core curve of the horizontal cylinder on the flat pillowcase. 

As  $\lambda\to 1$, the cover $\pi^{-1}(\lambda)$ converges to a cover of the this noded Riemann surface. This cover is $S(n,m)$ with all lifts of core curves of horizontal cylinders pinched. By Proposition \ref{P:kleinlift}, $\tilde{\sigma_2}$ preserves each of these curves while exchanging the two sides of the curve, so in the limit $\tilde{\sigma_2}$ fixes each node and exchanges the two sides of each node. Hence the quotient by $\tilde{\sigma_2}$ and also all of $\bZ_2\times \bZ_2$ is smooth. 

The situation is similar as $\lambda\to \infty$.  
\end{proof}

The upshot is that as, for example, $\lambda\to 1$ (or more accurately a lift of the puncture at $1$ in $B_0$ to $B$), the fiberwise quotients $\pi^{-1}(\lambda)/(\bZ_2\times\bZ_2)$ converge to a (smooth) Riemann surface, a point on the interior of moduli space. Hence, at lifts of $\lambda=1, \infty$ to $B$, the fiberwise quotient map has removable singularities. 



\section{$\T(n,m)$ is isometrically immersed}\label{S:BM}

In this section, we construct the Veech-Ward-Bouw-M\"{o}ller Teichm\"{u}ller curves $\T(n,m)\subset \M_g$, and prove that they are isometrically immersed.


\subsection{Base change.}\label{SS:BC} Given a family $\M$ over a base $B_0$, a \emph{base change} is the result of taking a finite cover $B\to B_0$, and pulling back to obtain a family $\M'$ over $B$.

Proposition \ref{P:kleinlift} guarantees that to each fiber of $\S(n,m)$ there is at least one and at most finitely many lifts $\langle \tilde{\sigma_2}, \tilde{\sigma_4}\rangle$ of the pillowcase symmetry group so that $\tilde{\sigma_2}, \tilde{\sigma_4}$ both have fixed points. 

Hence after applying a base change to the family $\S=\S(n,m)$ over the base $B_0=\CP\setminus\{0,1,\infty\}$, we may obtain a family $\S'$ over some larger base $B$ to which a continuous assignment of such lifts of the pillowcase symmetries is possible. 

As we will see, exactly what base change is required is not relevant to our arguments.


\subsection{Construction.} We define the fiberwise quotient map $q:B\to \M_g$ by $q(b)=\pi^{-1}(b)/(\bZ_2\times \bZ_2)$. The new base $B$ covers the old base $B_0=\CP\setminus\{0,1,\infty\}$; both are punctured Riemann surfaces. Proposition \ref{P:smooth} gives that the map $q:B\to \M_g$ has some removable singularities. We wish to ``fill in" these removable singularities, but there are some technical issues because $\M_g$ is an orbifold, and the image of the removable singularities might be orbifold points. 

For this reason we pass at once to a finite cover $\M_g'$ which is a manifold. We consider the minimal cover $B'\to B$ to which there is a map $B'\to \M_g'$ which covers the map $q:B\to\M_g$. We now consider the unique Riemann surface $\T'$ through which the map $B'\to \M_g'$ factors as
$$B'\to \T'\xrightarrow{\tau'} \M_g'$$
with $\tau$ generically one-to-one and without removable singularities. The procedure for producing $\T'$ from $B'$ is quite explicit: First, pass to the space covered by $B'$, from which the induced map to $\M_g'$ is generically one-to-one. Then, fill in the removable singularities to obtain $\T'$. 

The Riemann surface $\T'$ is equipped with the hyperbolic metric given by uniformization, and it is our goal to show that $\tau:\T'\to \M_g'$ is an isometry, showing that the induced generically one-to-one map $\T\xrightarrow{\tau} \M_g$ is a Teichm\"uller curve. We refer to $\T=\T(n,m)$ as the $(n,m)$ Veech-Ward-Bouw-M\"oller curve.

The orbifold structure on $\T$ is determined by the requirement that the natural branched cover $\T'\to\T$ is an isometry. The situation thus far is summarized in Figure \ref{F:smalldiagram}. 

\begin{figure} 
$$\xymatrix{
B' \ar[r] \ar[d] & \T' \ar[r]^{\tau'} \ar[d] & \M_g'\ar[d] \\
B \ar[r] & \T \ar[r]^{\tau} & \M_g\\
}$$
\caption{}
\label{F:smalldiagram}
\end{figure}

\subsection{Period mapping.} We will now translate our understanding of the period mapping of $\S$ to $\T$.  This begins with determining which part of $H^1(\pi^{-1}(b))$ survives after fiberwise quotient by $\bZ_2\times \bZ_2$.

\begin{lem}\label{L:decomp}
The complex VHS $H^1$ whose fibers are first cohomology over the family of fiberwise quotients $\S'/(\bZ_2\times \bZ_2)$ is isomorphic to $\oplus H^1(\br)$ where the sum runs over $\br$ in the row span of $S(n,m)$ with $t_1(\br)>\max(t_2(\br), t_3(\br), t_4(\br))$. 
\end{lem}

We will not enter into the definition of a complex VHS here, but rather remark that a complex sub-VHS of $H^1$ is merely a flat subbundle which splits into its $(1,0)$ and $(0,1)$ parts. Lemma \ref{L:decomp} asserts an isomorphism of complex VHS, which means an isomorphism of flat bundles which respects the Hodge decomposition into $(p,q)$ parts. 

The condition $t_1(\br)>\max(t_2(\br), t_3(\br), t_4(\br))$ is merely a way of picking a representative from each group of four $H^1(\br)$ which is permuted by the action of $\bZ_2\times \bZ_2$. 

\begin{proof}
For any $b\in B'$, the vector space $H^1(\pi^{-1}(b)/(\bZ_2\times \bZ_2))$ is isomorphic to the subspace of $H^1(\pi^{-1}(b))$ which is invariant under the Klein four group action.

Let $R$ be the set of $\br$ indicated in the lemma statement. Then, as in the proof of Lemma \ref{L:action}, we see that $R$ contains exactly one $\br$ for each orbit of size four of the Klein four group action on the set of nonzero $H^1(\br)$. 

By Lemma \ref{L:action} the set of invariants of $H^1(\pi^{-1}(b))$ is equal to $$\left\{\sum_R \sum_{\sigma\in \bZ_2\times \bZ_2} \sigma^*v_\br : v_\br\in H^1(\br), \br\in R\right\}.$$

Since the Klein four group preserves the complex structure on $S(n,m)$, it preserves the Hodge decomposition $H^1=H^{1,0}\oplus H^{0,1}$. The induced map $q^*$ on cohomology is thus an isomorphism of complex VHS from $H^1(\pi^{-1}(b)/(\bZ_2\times \bZ_2))$ to its image (the set of invariants) in $H^1(\pi^{-1}(b))$. 

\end{proof}

 The easiest way to describe the period (Abel-Jacobi, Torelli) map of $\T$ is to lift to Torelli space and consider the map to Siegel upper half space. 

\textbf{Torelli space and Siegel upper half space.} Let $\h_g$ denote Siegel upper half space, the space of $g\times g$ symmetric complex matrices whose imaginary part is positive definite. Let $\R_g$ be Torelli space, the quotient of Teichm\"uller space by the Torelli subgroup of the mapping class group. Equivalently, $\R_g$ is the set of Riemann surfaces $X$ with a choice of symplectic basis $\alpha_1, \beta_1, \ldots, \alpha_g, \beta_g$ of $H_1(X,\bZ)$. Given such a Riemann surface $X$, there is a unique basis $\omega_1, \ldots, \omega_g$ of holomorphic one forms which is dual to the $\alpha_i$, so $\int_{\alpha_i}\omega_j =\delta_{ij}$. Here $\delta_{ij}$ is 1 if $i=j$ and zero otherwise. The Riemann bilinear relations give that the period matrix $\left( \int_{\beta_j}\omega_i\right)_{ij}$ is symmetric and that its imaginary part is positive definite. The period map $p:\R_g\to \h_g$ sends a point in Torelli space to its period matrix. This map $p$ is holomorphic, and, by the Torelli theorem, locally injective. 

There are many lifts $\P$ of $\T'$ to $\R_g$ so that the following diagram commutes, one for each lift of a basepoint in $\T'$ to a basepoint in $\R_g$. 
$$\xymatrix{
\P\ar[r]\ar[d] & \R_g \ar[d]\\
\T'\ar[r]^{\tau'} & \M_g'\\
}$$
More precisely, by picking a symplectic basis of $H_1(\tau'(X))$ for any $X\in \T'$, we determine such a $\P$. 

The inclusion $\bH\subset \CP\setminus\{0,1,\infty\} = B_0$ can be lifted (not uniquely) to a holomorphic map $f:\bH\to \P$.

\begin{figure} 
$$\xymatrix{
&&&& (\bH)^g\ar@{^{(}->}[d] \\
& & \P \ar[urr]^p \ar[d] \ar[r]& \R_g \ar[d] \ar[r] &  \h_g \ar[dd]\\
& B' \ar[d] \ar[r] & \T'\ar[r]^{\tau'} \ar[d] & \M_g' \ar[d]  &\\
& B \ar[d]^\pi \ar[r] & \T\ar[r]^\tau & \M_g \ar[r]  & \A_g\\
\bH  \ar@{^{(}->}@/^2pc/[uuurr]^f \ar@{^{(}->}[r] & B_0  & &&\\
}$$
\caption{The period map $\R_g\to \h_g$ covers the map $\M_g\to \A_g$ which sends a Riemann surface to its Jacobian. Siegel upper half space $\h_g$ is the universal cover of the space of principally polarized abelian varieties $\A_g$. The cover $\P$ to $\T'$ or $\T$ will turn out to be the universal cover, but this is not yet evident. }
\label{F:diagram}
\end{figure}
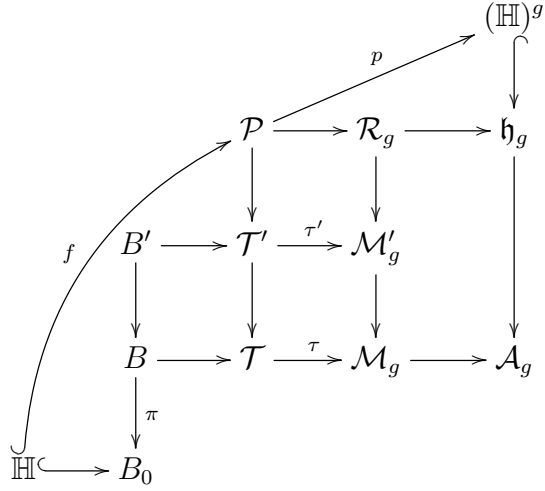 

Note that the set of diagonal matrices in $\h_g$ is naturally isomorphic to $(\bH)^g$. 

\begin{lem}
For any lift $\P$, the composition of $\P\to \R_g$ with the period mapping is diagonal, up to a change of basis which depends only on the choice of lift. Using this basis, we may thus write the period map $p=(p_1, \ldots, p_g)$, where the $p_j$ denote the diagonal entries. Furthermore, we may assume that the $p_j$ are in correspondence to the nonzero $\br$ in the row span of $S(n,m)$ with $t_1(\br)>\max(t_2(\br),t_3(\br),t_4(\br))$, and that for the $p_j$ corresponding to $\br$ the composite $p_j\circ f: \bH \to \bH$ is a Schwarz triangle map onto a triangle with angles $\kappa\pi, \mu\pi, \nu\pi$, where $$\kappa=|1-t_1-t_3|=0,\quad
\mu=|1-t_2-t_3|\in \frac{1}{m}\bZ,\quad
\nu=|1-t_1-t_2|\in \frac{1}{n} \bZ,$$
and $t_j=t_j(-\br)$. We may assume that $p_1$ corresponds to $$\br= (nm-n-m , nm+n-m , nm+n+m , nm-n+m)$$ and hence $p_1\circ f$ maps onto a triangle with angles $0, \frac{\pi}{m}, \frac{\pi}{n}$. 
\end{lem}

\begin{proof}
By Lemma \ref{L:decomp}, the bundle $H^1$ whose fibers are first cohomology over the family of fiberwise quotients $\S'/(\bZ_2\times \bZ_2)$ is isomorphic to $\oplus H^1(\br)$ where the sum runs over $\br$ in the row span of $S(n,m)$ with $t_1(\br)>\max(t_2(\br),t_3(\br),t_4(\br))$. The map $f:\bH\to \P$ arises from lifting a copy of $\bH\subset \CP$. 

Over a lifted copy of $\bH$ in $B$ we may pick a basis of homology of the fiberwise quotients as follows: let $\alpha_j$ and $\beta_j$ be a symplectic basis for the symplectic complement of the annihilator in $H_1$ of the corresponding $H^1(\br)$. Then the period mapping to $\h_g$ is diagonal, and the period coefficient $p_j$ is just the period mapping of Proposition \ref{P:periods}, and it follows from Proposition  \ref{P:periods} that $p_i\circ f$ is a Schwarz triangle mapping onto the triangle with angles $\kappa, \mu, \nu$. 
\end{proof}

{\footnotesize
\renewcommand{\arraystretch}{1.2}
\begin{center}
\begin{figure}[h!]\label{F:tables}
\begin{minipage}[t]{0.16\linewidth}
\begin{center}
  \begin{tabular}[t]{|c|c|}
    \hline
    \multicolumn{2}{|c|}{$\T(2,7)$}\\ \hline
    $(0, \frac37, \frac12)$& $\frac15$  \\ \hline
    $(0, \frac27, \frac12)$& $\frac35$ \\ \hline
    $\mathbf{(0, \frac17, \frac12)}$& $1$  \\ \hline
  \end{tabular}
\end{center}
\end{minipage}
\hspace{0.3cm}
\begin{minipage}[t]{0.16\linewidth}
\begin{center}
  \begin{tabular}[t]{|c|c|}
    \hline
     \multicolumn{2}{|c|}{$\T(2,8)$}\\ \hline
    $(0, \frac38, \frac12)$  & $\frac13$\\ \hline
    $\mathbf{(0, \frac18, \frac12)}$ & $1$ \\ \hline
  \end{tabular}
\end{center}
\end{minipage}
\hspace{0.3cm}
\begin{minipage}[t]{0.16\linewidth}
\begin{center}
  \begin{tabular}[t]{|c|c|}
    \hline
    \multicolumn{2}{|c|}{$\T(2,9)$}\\ \hline
    $(0, \frac49, \frac12)$ &$\frac17$\\ \hline
    $\mathbf{(0, \frac13, \frac12)}$&$\frac37$ \\ \hline
    $(0, \frac29, \frac12)$&$\frac57$ \\ \hline
    $\mathbf{(0, \frac19, \frac12)}$&$1$ \\ \hline
  \end{tabular}
\end{center}
\end{minipage}
\hspace{0.3cm} 
\begin{minipage}[t]{0.16\linewidth}
\begin{center}
  \begin{tabular}[t]{|c|c|}
    \hline
     \multicolumn{2}{|c|}{$\T(2,10)$}\\ \hline
    $(0, \frac3{10}, \frac12)$ & $\frac12$\\ \hline
    $\mathbf{(0, \frac1{10}, \frac12)}$ &$1$ \\ \hline
  \end{tabular}
\end{center}
\end{minipage}
\hspace{0.3cm} 
\begin{minipage}[t]{0.16\linewidth}
\begin{center}
  \begin{tabular}[t]{|c|c|}
    \hline
     \multicolumn{2}{|c|}{$\T(2,11)$}\\ \hline
    $(0, \frac5{11}, \frac12)$ &$\frac19$\\ \hline
    $(0, \frac4{11}, \frac12)$  &$\frac13$\\ \hline
    $(0, \frac3{11}, \frac12)$  &$\frac59$\\ \hline
    $(0, \frac2{11}, \frac12)$  &$\frac79$\\ \hline
    $\mathbf{(0, \frac1{11}, \frac12)}$  &$1$\\ \hline
  \end{tabular}
\end{center}
\end{minipage}

\begin{minipage}[t]{0.16\linewidth}
\begin{center}
  \begin{tabular}[t]{|c|c|}
    \hline
     \multicolumn{2}{|c|}{$\T(3,6)$}\\ \hline
    $(0, \frac16, \frac23)$ &$\frac13$\\ \hline
    $\mathbf{(0, \frac12, \frac13)}$&$\frac13 $ \\ \hline
    $(0, \frac13, \frac13)$&$\frac23 $ \\ \hline
    $\mathbf{(0, \frac16, \frac13)}$&$1$ \\ \hline
  \end{tabular}
\end{center}
\end{minipage}
\hspace{0.3cm} 
\begin{minipage}[t]{0.16\linewidth}
\begin{center}
   \begin{tabular}[t]{|c|c|}
    \hline
     \multicolumn{2}{|c|}{$\T(3,7)$}\\ \hline
    $(0, \frac27, \frac23)$&$\frac1{11}$ \\ \hline
    $(0, \frac47, \frac13)$&$\frac2{11}$ \\ \hline
    $(0, \frac17, \frac23)$&$\frac4{11}$ \\ \hline
    $(0, \frac37, \frac13)$&$\frac5{11}$ \\ \hline
    $(0, \frac27, \frac13)$&$\frac{8}{11}$ \\ \hline
     $\mathbf{(0, \frac17, \frac13)}$&$1$ \\ \hline
  \end{tabular}
\end{center}
\end{minipage}
\hspace{0.3cm} 
\begin{minipage}[t]{0.16\linewidth}
\begin{center}
  \begin{tabular}[t]{|c|c|}
    \hline
     \multicolumn{2}{|c|}{$\T(3,8)$}\\ \hline
    $(0, \frac58, \frac13)$ &$\frac1{13}$\\ \hline
    $(0, \frac14, \frac23)$ &$\frac2{13}$\\ \hline
    $\mathbf{(0, \frac12, \frac13)}$&$\frac4{13}$\\ \hline
    $(0, \frac18, \frac23)$ &$\frac5{13}$\\ \hline
    $(0, \frac38, \frac13)$ &$\frac7{13}$\\ \hline
    $\mathbf{(0, \frac14, \frac13)}$&$\frac{10}{13}$\\ \hline   
    $\mathbf{(0, \frac18, \frac13)}$ &$1$\\ \hline
  \end{tabular}
\end{center}
\end{minipage}
\hspace{0.3cm} 
\begin{minipage}[t]{0.16\linewidth}
\begin{center}
  \begin{tabular}[t]{|c|c|}
    \hline
     \multicolumn{2}{|c|}{$\T(3,9)$}\\ \hline
    $(0, \frac29, \frac23)$ &$\frac15$\\ \hline
    $(0, \frac59, \frac13)$&$\frac15 $\\ \hline
    $(0, \frac19, \frac23)$ &$\frac25 $\\ \hline
    $(0, \frac49, \frac13)$&$\frac25 $\\ \hline
    $\mathbf{(0, \frac13, \frac13)}$&$\frac35 $ \\ \hline 
    $(0, \frac29, \frac13)$&$\frac45 $\\ \hline
    $\mathbf{(0, \frac19, \frac13)}$&$1$ \\ \hline
  \end{tabular}
\end{center}
\end{minipage}
\hspace{0.3cm}
\begin{minipage}[t]{0.16\linewidth}
\begin{center}
  \begin{tabular}[t]{|c|c|}
    \hline
     \multicolumn{2}{|c|}{$\T(4,5)$}\\ \hline
    $(0, \frac15, \frac34)$ &$\frac1{11}$\\ \hline
    $(0, \frac25, \frac12)$ &$\frac2{11} $\\ \hline
    $(0, \frac35, \frac14)$&$\frac3{11} $\\ \hline
    $\mathbf{(0, \frac15, \frac12)}$&$\frac6{11} $ \\ \hline
    $(0, \frac25, \frac14)$ &$\frac7{11} $\\ \hline
    $\mathbf{(0, \frac15, \frac14)}$&$1$ \\ \hline
  \end{tabular}
\end{center}
\end{minipage}

\begin{minipage}[t]{0.16\linewidth}
\begin{center}
  \begin{tabular}[t]{|c|c|}
    \hline
     \multicolumn{2}{|c|}{$\T(4,6)$}\\ \hline
    $(0, \frac16, \frac34)$&$\frac17$ \\ \hline
    $\mathbf{(0, \frac13, \frac12)}$ &$\frac27 $\\ \hline
    $\mathbf{(0, \frac12, \frac14)}$ &$\frac37 $\\ \hline
    $\mathbf{(0, \frac16, \frac14)}$ &$1$\\ \hline
  \end{tabular}
\end{center}
\end{minipage}
\hspace{0.3cm} 
\begin{minipage}[t]{0.16\linewidth}
\begin{center}
  \begin{tabular}[t]{|c|c|}
    \hline
     \multicolumn{2}{|c|}{$\T(4,8)$}\\ \hline
    $(0, \frac18, \frac34)$&$\frac15$ \\ \hline
    $(0, \frac58, \frac14)$&$\frac15 $ \\ \hline
    $\mathbf{(0, \frac14, \frac12)}$&$\frac25 $ \\ \hline
    $(0, \frac38, \frac14)$&$\frac35 $ \\ \hline
    $\mathbf{(0, \frac18, \frac14)}$&$1$ \\ \hline
  \end{tabular}
\end{center}
\end{minipage}
\hspace{0.3cm} 
\begin{minipage}[t]{0.16\linewidth}
\begin{center}
  \begin{tabular}[t]{|c|c|}
    \hline
     \multicolumn{2}{|c|}{$\T(5,5)$}\\ \hline
    $(0, \frac15, \frac35)$ &$\frac13$\\ \hline
    $(0, \frac25, \frac25)$&$\frac13 $\\ \hline
    $(0, \frac35, \frac15)$ &$\frac13 $\\ \hline 
    $(0, \frac15, \frac25)$&$\frac23 $ \\ \hline
    $(0, \frac25, \frac15)$&$\frac23 $\\ \hline
    $\mathbf{(0, \frac15, \frac15)}$&$1$ \\ \hline
  \end{tabular}
\end{center}
\end{minipage}
\hspace{0.3cm}
\begin{minipage}[t]{0.16\linewidth}
\begin{center}
  \begin{tabular}[t]{|c|c|}
    \hline
     \multicolumn{2}{|c|}{$\T(6,10)$}\\ \hline
    $(0, \frac1{10}, \frac56)$&$\frac1{11}$ \\ \hline
    $(0, \frac35, \frac13)$ &$\frac1{11}$\\ \hline
    $(0, \frac7{10}, \frac16)$&$\frac2{11}$\\ \hline
    $(0, \frac15, \frac23)$&$\frac2{11}$ \\ \hline
    $(0, \frac3{10}, \frac12)$ &$\frac3{11}$\\ \hline
    $(0, \frac25, \frac13)$&$\frac4{11}$ \\ \hline
    $\mathbf{(0, \frac12, \frac16)}$&$\frac5{11}$ \\ \hline
    $\mathbf{(0, \frac1{10}, \frac12)}$&$\frac6{11}$\\ \hline
    $\mathbf{(0, \frac15, \frac13)}$&$\frac7{11}$ \\ \hline
     $(0, \frac3{10}, \frac16)$ &$\frac8{11}$\\ \hline
    $\mathbf{(0, \frac1{10}, \frac16)}$&$1$ \\ \hline
  \end{tabular}
\end{center}
\end{minipage}
\hspace{0.3cm}
\begin{minipage}[t]{0.16\linewidth}
\begin{center}
  \begin{tabular}[t]{|c|c|}
    \hline
     \multicolumn{2}{|c|}{$\T(8,8)$}\\ \hline
    $(0, \frac18, \frac58)$&$\frac13$ \\ \hline
    $(0, \frac38, \frac38)$&$\frac13 $ \\ \hline
    $(0, \frac58, \frac18)$&$\frac13 $ \\ \hline
    $\mathbf{(0, \frac14, \frac12)}$&$\frac13 $ \\ \hline
    $\mathbf{(0, \frac12, \frac14)}$&$\frac13 $ \\ \hline
    $\mathbf{(0, \frac14, \frac14)}$&$\frac23 $ \\ \hline
    $(0, \frac18, \frac38)$&$\frac23 $ \\ \hline
    $(0, \frac38, \frac18)$&$\frac23 $ \\ \hline
    $\mathbf{(0, \frac18, \frac18)}$&$1$ \\ \hline
  \end{tabular}
\end{center}
\end{minipage}
\caption{The first column contains a list of triples $(\kappa, \mu, \nu)$ for many $\T(n,m)$, one for each summand in the decomposition $H^1=\oplus H^1(\br)$ given in Lemma \ref{L:decomp}. The bold entries correspond to triangles which tile the hyperbolic plane. They correspond to the $\T(n',m')$ in Theorem \ref{T:cov}. The second column gives the Lyapunov exponents, calculated in Theorem \ref{T:BMlyaps}. The number of columns is the genus of $\T(n,m)$. Many results in this paper can be conjectured from these tables. 
} 
\label{F:tables}
\end{figure}
\end{center}
}

\begin{lem}
There is a covering $c: \P\to \CP\setminus\{\infty\}$, branched only over $\{0,1\}$ so that for any lift $f:\bH\to \P$ of the inclusion $\bH\subset \CP\setminus\{\infty\} $, the composite $c\circ f: \bH\to\bH$ is the identity. 
\end{lem}

Essentially this technical lemma says that $\P$ is naturally a cover of the base $B_0$, with points over $\lambda=0,1$ filled in. 

\begin{proof}
Note that $p_1$ induces a map $$c:\P \to \bH/\Delta^+(n,m,\infty)\simeq \CP\setminus\{\infty\}.$$ The isomorphism can be chosen so that any lift $f:\bH\to \P$ composed with this composite is the identity map from the upper half plane to itself. 

The key point now is that the preimage under $c$ of the cone points of $\bH/\Delta^+(n,m,\infty)$ are not ``missing" from $\P$. This follows directly from Proposition \ref{P:smooth} as follows. 

Suppose we fix a lift $f:\bH\to \P$ of the inclusion $\bH\subset B_0$. We must show that this map can be extended continuously to $0,1$. Since $\P$ covers $\T'$, it suffices to prove this for lifts $f:\bH\to \T'$. In this case $f(0)$ corresponds to the fiberwise quotient of some noded Riemann surface, which is smooth by Proposition \ref{P:smooth}. Hence the fiberwise quotient map has a removable singularity at this point. By the definition of $\T'$, all removable singularities are filled it. 

To avoid confusion, we remark parenthetically that it is not possible for any punctures over $\infty$ to get filled in since the period map is not proper at these punctures. (That is, points over $\lambda=\infty$ get mapped to the cusp of the Schwarz triangle via the period map $p_1$.) 
\end{proof}

\subsection{Two elementary lemmas.} The next pair of lemmas are elementary, in that they not have anything to do with moduli spaces. For the first, recall that if $f:\bC\to \bC$ is holomorphic, for any point $z_0\in \bC$, there is a holomorphic change of coordinates in the domain sending $z_0$ to $0$, and a holomorphic change of coordinates in the range, so that $f(z)=z^d$. We define $d$ as the \emph{degree of $f$ at $z_0$}. 

\begin{lem}
Let $U$ be a neighborhood of $0$ in $\bC$, and $p:U\to \bC^g$ be a holomorphic map with coordinates $p=(p_1, \ldots, p_g)$. Suppose that $p$ is locally injective at 0, and that the degree of $p_1$ at 0 divides the degree of $p_i $ at 0 for each $i=2,\ldots, g$. Then $p_1'(0)\neq 0$. 
\end{lem}
\begin{proof}
By a holomorphic change of coordinates, we may assume $p_1(z)=z^k$ and $p_i(z)=z^{kd_i}$ for some $d_i\geq 1$. It follows immediately from local injectivity that $k=1$. 
\end{proof}

\begin{lem}
Let $P$ be a Riemann surface which covers $\CP\setminus\{\infty\}$, branched over at most $0$ and $1$. Suppose we are given a holomorphic locally injective map $p: P \to (\bH)^g, p=(p_1, \ldots, p_g)$, such that for any lift $f:\bH\to P$ of the inclusion $\bH\subset \CP\setminus\{\infty\}$,  
\begin{itemize}
\item for each $i=1,\ldots, g$, the composite $p_i\circ f: \bH\to \bH$ is a biholomorphism onto a hyperbolic triangle with angles $\kappa\pi=0, \mu\pi\in \frac{\pi}{m}\bZ, \nu\pi\in \frac{\pi}{n}\bZ$ at  $p_i\circ f (0), p_i\circ f(1), p_i\circ f(\infty)$, and 
\item when $i=1$, this triangle has angles $\kappa\pi=0, \mu\pi= \frac{\pi}{m}, \nu\pi= \frac{\pi}{n}$. 
\end{itemize}
Then $P=\bH$ and $p_1:\bH\to \bH$ is an isometry. 
\end{lem}

The intuition is that $f(\bH)\subset P$ is a (a priori possibly very squiggly) triangle, whose angles are determined by the requirement that $p$ is holomorphic and locally injective (see figure \ref{F:squiggles}). If the angles of this triangle are  $\kappa\pi=0, \mu\pi= \frac{\pi}{m}, \nu\pi= \frac{\pi}{n}$, then such triangles tile $P$ and $p_1$ is an isometry. 

\begin{proof}
We begin by tiling the codomain $\bH$ of $p_1$ with reflected copies of the geodesic triangle $p_1(f(\bH))$. Suppose $p_1$ is locally $k$ to 1 at $f(1)$. By pulling back the local picture at $p_1(f(1))$, we see that the angle at $f(1)$ is of the form $\frac{\pi}{km}$(figure \ref{F:squiggles}). The local picture at $f(1)$ consists of $2km$ triangle segments, pulled back from the $2m$ triangle segments around $p_1(f(1))$. The degree of $p_1$ at $f(1)$ is $k$, and the conditions of the lemma gives that the degree of $p_i$ at $f(1)$ is a multiple of $k$ (for $i=2,\ldots, g$). The previous lemma gives that $k=1$, so the angle at $f(1)$ is $\frac{\pi}{m}$. Similarly we may see that the angle at $f(\infty)$ is $\frac{\pi}{n}$.

\begin{figure}[h]
\includegraphics[scale=0.5]{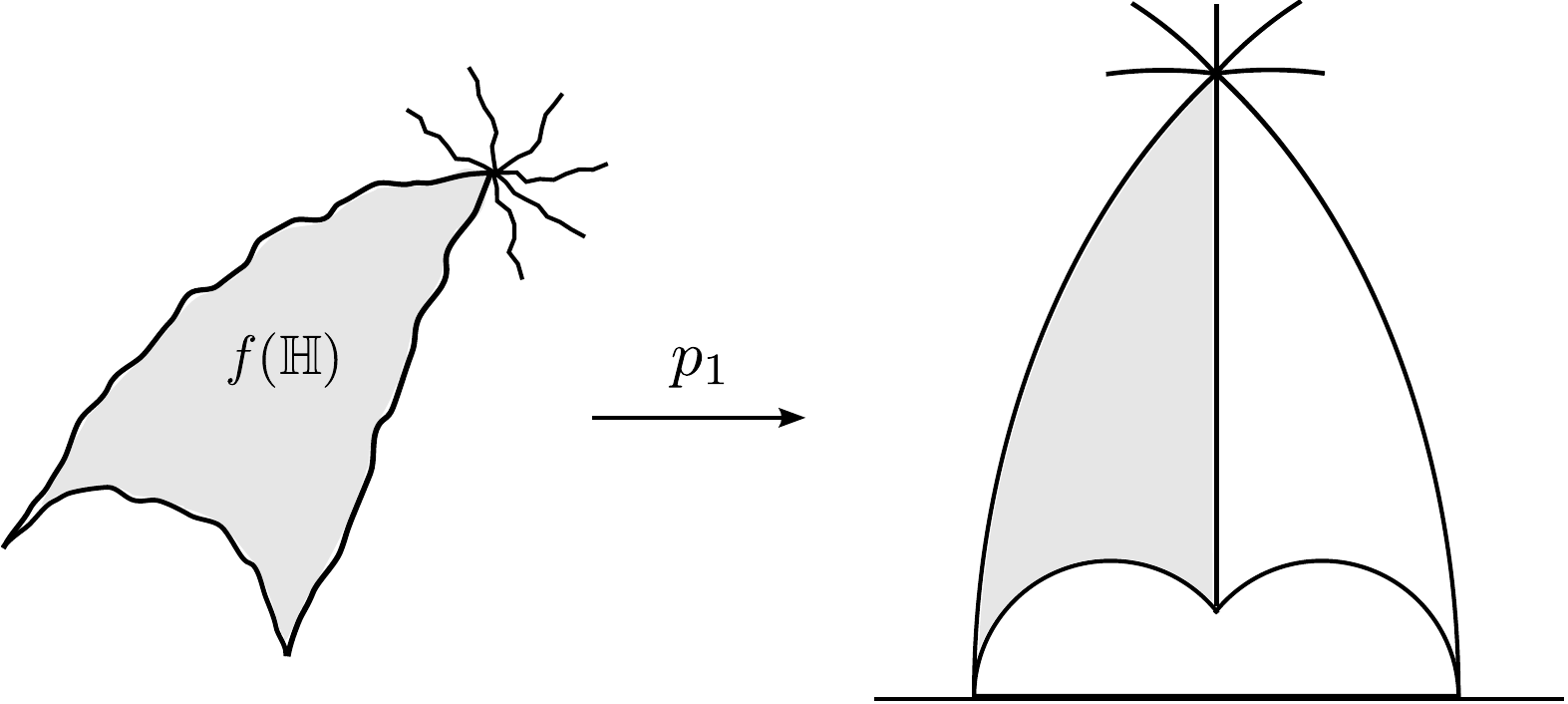}
\caption{$f(\bH)$ should be thought of as a possibly very squiggly triangle in the Riemann surface $P$. It maps via $p_1$ to a geodesic triangle, and the local picture of the tiling by such geodesic triangles at a corner can be pulled back to a corner of  $f(\bH)$.}
\label{F:squiggles}
\end{figure}

We have shown that $f(\bH)$ has angles $\frac{\pi}{m}, \frac{\pi}{n}$. We may pull back the triangles bordering $p_1(f(\bH))$ to obtain (again, a priori squiggly) triangles ``next" to $f(\bH)$. We say ``next" despite the fact it is not immediately obvious that the triangles next to $f(\bH)$ are disjoint from $f(\bH)$. However, the usual argument that the geodesic triangle with angles $0, \frac{\pi}{m}, \frac{\pi}{n}$ tiles $\bH$ applies here, and we see that $P$ is tiled by preimages under $p_1$ of triangles in $\bH$. We also see that $p_1$ is a proper holomorphic map with non vanishing derivative onto $\bH$.  Any covering map onto $\bH$ is an isomorphism since $\bH$ is simply connected, hence $P=\bH$. Any biholomorphism of $\bH$ is an isometry, so $p_1$ is an isometry. 
\end{proof}

\subsection{The magic of the Kobayashi metric.}

\begin{thm}\label{T:BM}
$\T$ is a Teichm\"uller curve. 
\end{thm}

\begin{proof}
The period coefficient $p_1$ is defined on all of Torelli space $\R_g$. It is simply the top left entry in the period matrix. Since the period matrix has positive definite imaginary part, $p_1$ has positive imaginary part. So $p_1:\R_g\to \bH$.  

The previous two lemmas show that the composition of the inclusion of $\P$ into Torelli space $\R_g$ with the period coefficient $p_1$ is an isometry. A theorem of Royden gives that the Teichm\"uller metric on $\M_g$ and hence also $\R_g$ is equal to the Kobayashi metric. Since holomorphic maps are distance nonincreasing, the composition of two maps is a Kobayashi isometry if and only if each map is. Hence the inclusion $\P$ into Torelli space is an isometry from $\P$ with its hyperbolic metric to $\R_g$ with the Teichm\"uller metric. Since this inclusion covers the map $\T\to\M_g$, we see that this last map is an isometry also. 
\end{proof}

\subsection{Lyapunov exponents.}\label{SS:BMlyaps} In \cite[Corollary 6.9]{BM} the Lyapunov exponents of $\T(n,m)$ are computed when $n$ and $m$ are odd and relatively prime. Here we take a different approach, using period mappings as in our treatment of abelian square-tiled surfaces in \cite{W1}.

\begin{thm}\label{T:BMlyaps} 
The nonnegative part $\Lambda$ of the Lyapunov spectrum of $H^1$ for $\T(n,m)$ is given by the following algorithm. Start with $\Lambda=\emptyset$, and for every $\br$ in the row span of $S(n,m)$ with $t_1(\br)>\max(t_2(\br), t_3(\br), t_4(\br))$, add 
$$\lambda=\frac{2\min_{j=1,2,3,4} \{t_j(-\br), 1-t_j(-\br)\}}{1-\frac1{n}-\frac1{m}}$$
to $\Lambda$. 

The quantity $\lambda$ is the ratio of the hyperbolic area of the Schwarz triangle describing the period map of $H^1(\br)$ over the hyperbolic area of the triangle with angles $0, \frac{\pi}n, \frac{\pi}m$. 
\end{thm}

\begin{proof}
The proof proceeds exactly as in \cite[Section \ref{S:lyaps}]{W1} for abelian square-tiled surfaces. Specifically, based on \cite{Fex1, EKZbig} it is computed in \cite[Theorem \ref{T:periodlyap}]{W1} that the Lyapunov exponent is the average squared hyperbolic norm of the derivative of the period map. Then, as in \cite[Theorem \ref{T:arealyaps}]{W1}, the change of variables formula reveals that this is the ratio of the hyperbolic area of a triangle in the domain (in this case, the triangle with angles $0, \frac{\pi}n, \frac{\pi}m$) to the area of the image of this triangle under the Schwarz triangle map. In \cite[Theorem \ref{T:arealyaps}]{W1}, the area of the image triangle (that is, the triangle in the description of the Schwarz triangle mapping), is computed to be $2\pi\min_{j=1,2,3,4} \{t_j(-\br), 1-t_j(-\br)\}$. 
\end{proof}

Corollary \ref{C:Lmults} follows immediately. 


\section{Covering relations between $\T(n,m)$ and $\T(n',m')$}\label{S:rels} In this section we prove Theorem \ref{T:cov}. 

\subsection{Covering relations between $S(n,m)$ and $S(n',m')$.} Recall that $S(n,m)$ is defined as 

$$M_{2nm}\left(\begin{array}{cccc}
nm-n-m & nm+n-m & nm+n+m & nm-n+m \\
nm+n-m & nm-n-m & nm-n+m & nm+n+m\\
\end{array}\right).$$

Recall that the \emph{row span} of $S(n,m)$ is the abelian subgroup of $(\bZ/(2nm\bZ))^4$ generated the two row vectors above. Multiplying by any $k\in \bN$ allows us to also consider the row span as a subgroup of $(\bZ/(k\cdot2nm\bZ))^4$. 

Basic results in \cite[Section \ref{SS:order}]{W1} give that $S(n,m)$ covers $S(n',m')$ (in a way compatible with the maps to $\CP$), if $nm=k\cdot n'm'$ and $k$ times the row span of $S(n',m')$ is contained in the row span of $S(n,m)$. 

\begin{lem} 
The row span of $S(n,m)$ contains $$(2m, 2m, -2m, -2m)\quad\text{and}\quad (2n, -2n, -2n, 2n).$$ If either $n$ or $m$ is odd, it contains  $$(m, m, -m, -m)\quad\text{and}\quad (n, -n, -n, n).$$ In all cases it contains $$(-n-m, n-m, n+m, -n+m).$$
\end{lem} 

\begin{proof}
Left to the reader.
\end{proof}

\begin{lem}\label{L:cov}
If $n$ or $m$ is odd: $S(n',m')$ covers $S(n,m)$ if $n'$ divides $n$ and $m'$ divides $m$. 

If $n$ and $m$ are both even: $S(n',m')$ covers $S(n,m)$ if $n'$ divides $n$ and $m'$ divides $m$ and $n/n'+m/m'$ is even 
\end{lem}

\begin{proof}
Say $km'=m$, $ln'=n$. Multiplying all the rows of $S(n/l, m/k)$ by $kl$ we obtain 
$$\left(\begin{array}{cccc}
nm-kn-lm & nm+kn-lm & nm+kn+lm & nm-kn+lm \\
nm+kn-lm & nm-kn-lm & nm-kn+lm & nm+kn+lm\\
\end{array}\right).$$
It suffices to show that the two rows of the above matrix are contained in the row span of $S(n,m)$ \cite[Section \ref{SS:order}]{W1}. The first row differs from the first row of $S(n,m)$ by 
$$(-(k-1)n-(l-1)m , +(k-1)n-(l-1)m , \ldots , \ldots).$$
The parity of $(k-1)+(l-1)$ is the same as that of $n/n'+m/m'$. The result now follows from the previous lemma.  
\end{proof}

\subsection{From $S(n,m)$ to $\T(n,m)$.} We proceed to the proof that, under the conditions of Lemma \ref{L:cov}, every point of $\T(n,m)$ covers a point on $\T(n',m')$. 

\begin{proof}[\textbf{\emph{Proof of Theorem \ref{T:cov}.}}]
Say $km'=m$, $ln'=n$. Let $$\br'= (nm-kn-lm , nm+kn-lm , nm+kn+lm , nm-kn+lm).$$ The assumptions above guarantee that $\br'$ is in the row span of $S(n,m)$. 

We know already that $S(n,m)$ covers $S(n',m')$. Hence, the function field of $S(n',n')$ is contained in that of $S(n,m)$. Let $H'$ be the subgroup of the deck group of $S(n,m)$ which acts trivially on the function field of $S(n',m')$. So $S(n,m)/H=S(n',m')$. 

Note that we know exactly what each $\tilde{\sigma}H^1(\br')$ is by Lemma \ref{L:action}. We claim that $H'$ is the intersection of the stabilizers of $\tilde{\sigma}H^1(\br')$ for each $\tilde{\sigma}\in \bZ_2\times\bZ_2$.

It follows from this definition that $H'$ is invariant under conjugation by $\bZ_2\times \bZ_2$. From general principles, we may conclude that there is an induced action of $\bZ_2\times \bZ_2$ on $S(n',m')=S(n,m)/H'$. The map from $S(n,m)$ to $S(n',m')$ commutes with the $\bZ_2\times \bZ_2$ actions, so we get a map $$S(n,m)/(\bZ_2\times\bZ_2)\to S(n',m')/(\bZ_2\times\bZ_2).$$

As a technical detail, we should point out that the induced action of $\bZ_2\times\bZ_2$ on $S(n',m')$ necessarily has fixed points since the action of $\bZ_2\times \bZ_2$ on $S(n,m)$ does. 

This proves the result for all but the finitely many points of $\T(n,m)$ which are not in the image of the fiberwise quotient map. By Proposition \ref{P:smooth}, these are the lifts of of the punctures $1,\infty$ of $B_0$ to $B$. The result at these points follows by a continuity argument.
\end{proof}

\begin{rem}
The covering of the previous theorem is almost never a regular (normal, Galois) cover. The exception is that when $S(n,m)$ is a twofold cover of $S(n',m')$, the cover must be regular since all twofold covers are regular. Every fiber of $\T(n,m)$ admits an involution which negates the generating abelian differential, and when $S(n,m)$ is a twofold cover of $S(n',m')$, then the cover of Theorem \ref{T:cov} is the quotient by this involution. The involution is induced by the deck transformation $\left(\begin{array}{c} nm\\ nm \end{array}\right)$ on $S(n,m)$ and is visible as $y\mapsto -y$ in Theorem \ref{T:eqns}. This involution could presumably be used to \emph{twist} some or all of the $\T(n,m)$ as in \cite{Mc2} to obtain quadratic differentials which generate Teichm\"uller curves.  
\end{rem}


\section{Real multiplication of Hecke type}\label{S:Hecke}

In this section we prove Theorem \ref{T:Hecke}. We begin by recalling the definition of endomorphisms of Hecke type from \cite{E}. 

Given a compact Riemann surface $Y$, its Jacobian is defined as $$\Jac(Y)=H^{1,0}(Y)^*/H_1(Y,\bZ).$$ If $Y$ has endomorphism group $G$, then there is a natural map $\bQ[G]\to \End(\Jac(Y))\otimes \bQ$. Simply put, each automorphism $g\in G$ induces an an action $g_*$ on $H_1(Y, \bZ)$ which extends to a complex linear self-adjoint endomorphism of  $H^{1,0}(Y)^*$. This determines an endomorphism on $\Jac(Y)$. 

Now, let $H$ be a subgroup of $G$ and set $X=Y/H$. Quite possibly, $X$ has no automorphisms. Let $$\pi_H=\frac1{|H|}\sum_{h\in H} h_*$$ be the projection of $H_1(Y,\bZ)$ onto $H_1(X, \bZ)$. For every $g\in G$, we obtain a map 

\begin{equation}\label{E:pigpi}
 \pi_H \circ g_*\circ \pi_H:  H_1(X,\bZ)\to H_1(X, \bZ).
\end{equation} 

This induces a map from the double coset algebra $\bQ[H\backslash G/H]$ to the endomorphism algebra $\End(\Jac(X))\otimes \bQ$. Endomorphism in the image of this map (for some $Y, G, H$) are said to be of \emph{Hecke type}. See the introduction to \cite{E} for more details. 

To prove Theorem \ref{T:Hecke}, we will let $Y=S(n,m)$, and set $G$ to be the group of automorphisms of $S(n,m)$ generated by the abelian deck group and $\bZ_2\times \bZ_2$. We will let $H= \bZ_2\times \bZ_2$, so $X=Y/H$ is a point on $\T(n,m)$. As before, we note that all but finitely many points on $\T(n,m)$ are of this form. 

Our starting point is the following result of M\"oller \cite[Theorem 2.7]{M}.

\begin{thm}\label{T:MRM}
Let $\T$ be a Teichm\"uller curve generated by an abelian differential, and $F$ the trace field of its uniformizing group. For every point $X\in\T$, the Jacobian $\Jac(X)$ splits up to isogeny as $A_1\times A_2$, where $A_1$ admits real multiplication by $F$. The abelian differential $\omega$ on $X$ which generates the Teichm\"uller curve is an eigenform for the real multiplication.  
\end{thm}

Consider $X\in \T(n,m)$ of the form $X=S(n,m)/(\bZ_2\times\bZ_2)$ with generating differential $\omega$. To show that the real multiplication present on a factor of $\Jac(X)$ is of Hecke type, it suffices to exhibit maps of the form \ref{E:pigpi} with $\omega$ as an eigenform, for a set of eigenvalues which generate the trace field. Let us now collect the required results to do this. 

The uniformizing group of $\T(n,m)$ is commensurable to $\Delta^+(n,m,\infty)$, which has invariant trace field 
$$F=\bQ[\xi_{n}+\xi_{n}^{-1}, \xi_{m}+\xi_{m}^{-1}, (\xi_{2n}+\xi_{2n}^{-1})(\xi_{2m}+\xi_{2m}^{-1})],$$
see \cite{MR} and Section \ref{SS:tracef} below. For any uniformizing group of a Teichm\"uller curve, the trace field is equal to the invariant trace field \cite{KS}.

The proof of Proposition \ref{P:kleinlift} begins with the following elementary \emph{commutation relations}, which follow by general principles from the existence of fixed points for $\tilde{\sigma_2}$ and $\tilde{\sigma_4}$. Let $T_j$ be the deck transformation of $S(n,m)$ corresponding to the oriented loop $\alpha_j$ about $z_j$.  
 \begin{eqnarray}\label{E:intertwine}
 \tilde{\sigma_2} T_1 \tilde{\sigma_2} = T_2 \quad \quad \quad \tilde{\sigma_2} T_3 \tilde{\sigma_2} = T_4 \\  \tilde{\sigma_4} T_1 \tilde{\sigma_4} = T_4 \quad \quad \quad \tilde{\sigma_4} T_2 \tilde{\sigma_4} = T_3
 \end{eqnarray}

We may now proceed to proof, following the strategy outlined above.

\begin{proof}[\textbf{\emph{Proof of Theorem \ref{T:Hecke}.}}]
Take $\omega\in H^{1,0}(\br)$. Using the commutation relations we compute the action of \ref{E:pigpi} on $\omega$ when $g= T_1^p T_2^q$.
\begin{eqnarray*}
&&\left(\sum_{\sigma\in \bZ_2\times\bZ_2} \sigma\right) T_1^p T_2^q \left(\sum_{\sigma\in \bZ_2\times\bZ_2} \sigma\right) \omega
\\&=& 
\left( e^\frac{2\pi i (pr_1+qr_2)}{N}+e^\frac{2\pi i (-pr_1-qr_2)}{N}+e^\frac{2\pi i (pr_2+qr_1)}{N}+e^\frac{2\pi i (-pr_2-qr_1)}{N} \right)\left(\sum_{\sigma\in \bZ_2\times\bZ_2} \sigma\right)\omega
\end{eqnarray*}

Take $r_1=nm-n-m, r_2=nm+n-m$, so $\sum_{\sigma\in \bZ_2\times\bZ_2} \sigma\omega$ projects to a generator of $\T(n,m)$. Taking $p=q=1$, $p=-q=1$ and $p=1, q=0$, we see that the scalars 
$$\left( e^\frac{2\pi i (pr_1+qr_2)}{N}+e^\frac{2\pi i (-pr_1-qr_2)}{N}+e^\frac{2\pi i (pr_2+qr_1)}{N}+e^\frac{2\pi i (-pr_2-qr_1)}{N} \right)$$
generate the invariant trace field of $\Delta^+(n,m,\infty)$.
\end{proof}

\begin{rem}
Above we assume the existence of the real multiplication and then show it is of Hecke type. With only slightly more effort it is possible to show directly that the Jacobian splits up to isogeny and to establish Theorem \ref{T:MRM} directly. 
\end{rem}


\section{Comparison to Veech, Ward and Hooper}\label{S:grn}

In this section we give the generators for $\T(n,m)$. The following theorem may be viewed as an extension of the computations in \cite[Theorem 6.14]{BM}, which establishes the result when $n$ and $m$ are relatively prime and $n$ is odd; however, our methods are are somewhat different. We prove Theorem \ref{T:eqns} in Section \ref{S:eqns}. 

\begin{thm}\label{T:eqns}
$\T(n,m)$ is generated by the Riemann surface and holomorphic differential indicated below.

If $m$ is odd:
$y^{2n}=(u-2)\prod_{j=1}^{\frac{m-1}2} \left(u-2\cos\frac{2\pi j}{m}\right)^2,$\\ and $\frac{ydu}{(u-2)\prod_{j=1}^{\frac{m-1}2} \left(u-2\cos \frac{2\pi j}{m}\right)}.$

If $m$ is even and $n$ is odd:
$y^{2n}=(u-2)^{n}\prod_{j=1}^{\frac{m}2} \left(u-2\cos\frac{\pi (2 j-1)}{m}\right)^2,$ and $\frac{ydu}{(u-2)\prod_{j=1}^{\frac{m}2} \left(u-2\cos \frac{\pi (2j-1)}{m}\right)}.$ 

If $m$ and $n$ are both even:
$y^{n}=(u-2)^{\frac{n}2}\prod_{j=1}^{\frac{m}2} \left(u-2\cos\frac{\pi (2 j-1)}{m}\right),$ and $\frac{ydu}{(u-2)\prod_{j=1}^{\frac{m}2} \left(u-2\cos \frac{\pi (2j-1)}{m}\right)}.$
\end{thm}

From this description we immediately get Theorem \ref{T:same}, which says that Hooper's lattice surfaces generate the Veech-Ward-Bouw-M\"oller curves. 

\begin{proof}[\textbf{\emph{Proof of Theorem \ref{T:same}}}]
The generators given in Proposition \ref{T:eqns} coincide with those given in \cite[Proposition 17]{H}.
\end{proof}

Theorem \ref{T:same} allows us to immediately appeal to a Hooper's work \cite{H}, which was facilitated by his beautiful description of the flat structure of these generators in terms of semiregular polygons. 

\begin{proof}[\emph{\textbf{Proof of Corollary \ref{C:hoop}}}]
That $\T(n,m)=\T(m,n)$ follows from symmetry in the grid graph construction of \cite{H}.
The list of arithmetic $\T(n,m)$ is \cite[Corollary 10]{H} and follows from the calculation of the uniformizing group in \cite[Theorem 9]{H}. Geometric primitivity is \cite[Theorem 11]{H}. The genus computations are \cite[Theorems 13, 14]{H} and could moreover be easily verified from the generators above.
\end{proof}

\begin{rem}
The fundamental theorem of Schwarz-Christoffel mappings gives that the function 
$$f(z)=\int^z \prod_{k=1}^{p-1} (u-u_k)^{\alpha_k-1} du$$
maps the upper half-plane $\bH$ to a polygon (which may have self intersections) with angles $\pi \alpha_1, \ldots, \pi\alpha_p$, at $f(u_1), \ldots, f(u_{p-1}), f(\infty)$ (we assume $0< \alpha_k<2$). 

For the one forms in \ref{T:eqns}, if $m=2$ then the Schwarz-Christoffel mapping maps on a triangle with angles $\frac{\pi}{2}, \frac{\pi}{n}, \frac{(n-2)\pi}{2n}$. If $m=3$, it maps onto a triangle with angles $\frac{\pi}{2n}, \frac{\pi}{n}, \frac{(2n-3)\pi}{2n}$. It follows that $\T(2,n)$ is the Veech Teichm\"uller curve generated by the regular $n$--gon, and that $\T(3,m)$ is the Ward Teichm\"uller curve arising from billiards in the triangle with angles $\frac{\pi}{2n}, \frac{\pi}{n}, \frac{(2n-3)\pi}{2n}$ \cite{V, W}. 

See \cite[Corollary 5.6, 6.14]{BM} for a more algebraic discussion of these coincidences. Lochak has explicitly given the family for the Veech curves \cite{L}, see also \cite{Mc4}.

When $m=4,5$, in \cite{BM} it was determined that the image of the Schwarz-Christoffel mapping does not have self crossing, and quadrilateral billiards which unfold to the above generators were explicitly determined.
\end{rem}


\section{Algebraic primitivity}\label{S:algprim}

Hooper has shown that the $(n,m)$ Veech-Ward-Bouw-M\"oller curve $\T(n,m)$ is always geometrically primitive (Corollary \ref{C:hoop}). Recall that a Teichm\"uller curve is called algebraically primitive if the degree of the trace field is equal to the genus \cite{M2}. Algebraic primitivity implies geometric primitivity, but not vice versa \cite{M2}. 

The goal of this section is to prove Theorem \ref{T:ap} which gives when $\T(n,m)$ is algebraically primitive. We give a formula for the degree of the (invariant) trace field of $\T(n,m)$, so the proof of Theorem \ref{T:ap} could proceed by directly comparing the degree of the trace field to the genus. This is doable, but involves rather a lot of cases, so we begin by reducing the number of cases to be checked.

We also recover a result of Hooper showing that not all triangle groups uniformize Teichm\"uller curves. 


\subsection{An obstruction to algebraic primitivity.}\label{SS:notap} Here we show

\begin{prop}
If $S(n,m)$ nontrivially covers some $S(n', m')$ with $n', m'>1, n'm'\geq 6$, then $\T(n,m)$ is not algebraically primitive. 
\end{prop}

The non-triviality requirement simply means that $(n,m)\neq(n',m')$. 

\begin{cor}\label{C:notap}
Unless $n=m=4$, or $n$ and $m$ are odd primes, or one of $n$ and $m$ is 2 and the other is a power of two, a prime, or twice a prime, then $\T(n,m)$ is not algebraically primitive. 
\end{cor}

\begin{proof}[\textbf{\emph{Proof of Corollary.}}]
The proof proceeds by cases using Lemma \ref{L:cov}.

\textbf{$n,m>2$ even} (excluding $(4,4)$): $S(n,m)$ covers $S(n/2, m/2)$.

\textbf{$n,m$ not both even, not both prime:} say $m=km'$. Then $S(n,m)$ covers $S(n, m')$. 

\textbf{$n=2$, $m$ even}: write $m=ab$, with $a=2^p$ and $b$ odd. If $b>1$ and $a>2$, then $S(n,m)$ covers $S(n, a)$. If $a=2$ and $b=cd$ is not prime, then $S(n,m)$ covers $S(n, 2c)$.
\end{proof}

\begin{proof}[\textbf{\emph{Proof of Proposition.}}]
When the period mapping of $H^1(\br)$ is described by a Schwarz triangle, the monodromy of this bundle is commensurable to the corresponding triangle group. 
If $S(n,m)$ covers $S(n',m')$, then $\T(n,m)$ has a $H^1(\br)$ whose period mapping is given by a triangle with angles $\frac \pi{n'}, \frac \pi{m'}, 0$. 
Hence $\T(n,m)$ has a rank two local subsystem with monodromy commensurable to $\Delta^+(n',m',\infty)$. However, an algebraically primitive Teichm\"uller curve may only have one rank two local subsystem with discrete monodromy \cite{M}. 
\end{proof}


\subsection{Trace fields.}\label{SS:tracef}
For any integer $k$, let $\xi_k=\exp(2\pi i/k)$. Set 
\begin{eqnarray*}
F&=&\bQ[\xi_{2n}+\xi_{2n}^{-1}, \xi_{2m}+\xi_{2m}^{-1}],\\
 E&=& \bQ[\xi_{n}+\xi_{n}^{-1}, \xi_{m}+\xi_{m}^{-1}, (\xi_{2n}+\xi_{2n}^{-1})(\xi_{2m}+\xi_{2m}^{-1})].
\end{eqnarray*}
So $F$ is the trace field of the $(n,m,\infty)$ triangle group, and $E$ is the invariant trace field \cite{MR}. Set $\gamma=\gcd(n,m)$ and $l=\lcm(n,m)$. Note that both $E$ and $F$ are normal subfields of the cyclotomic extension $\bQ[\xi_{2l}]$ over $\bQ$. Recall that $Gal_{\bQ}(\bQ[\xi_{2l}])=\bZ_{2l}^*$, and that $a\in \bZ_{2l}^*$ acts by $a(\xi_{2l})=\xi_{2l}^a$. 

There is a map from $Gal_\bQ(\bQ[\xi_{2l}]) \to F$, and the size $s$ of the kernel is equal to the degree of $\bQ[\xi_{2l}] $ over $F$. Hence the degree of $F$ over $\bQ$ is $\varphi(2l)/s$, where $s$ is easily computed. The degree of $E$ may similarly be computed. 

\begin{prop}\label{P:tracef}
The degree of $F$ over $\bQ$ is $\phi(2l)/4$ if $\gamma=1$, and $\phi(2l)/2$ otherwise. 
\end{prop}

%

Note that if either $n$ or $m$ is odd, then we may directly see that $E=F$.  

\begin{prop}\label{P:invtracef}
Assume that $n$ and $m$ are even. The degree of $E$ over $\bQ$ is $\varphi(2l)/4$, unless $\gamma>2$ and one of $n/\gamma$ or $m/\gamma$ is even, in which case the degree is $\varphi(2l)/2$.
\end{prop}

We may now pause to recover a result of Hooper \cite[Theorem 2]{H}.

\begin{cor}\label{C:deg2ext}
The invariant trace field and the trace field of the $(n,m,\infty)$ triangle group have the same degree, unless $n$ and $m$ are even, and either $\gamma=2$, or both $n/\gamma$ and $m/\gamma$ are odd, in which case the trace field is a degree two extension of the invariant trace field. In the latter cases, there is no Teichm\"uller curve uniformized by $\Delta^+(n,m,\infty)$.
\end{cor}

\begin{proof}
Kenyon-Smillie have shown that for the uniformizing group of a Teichm\"uller curve, the trace field and invariant trace field always coincide \cite{KS}. 
\end{proof}


\subsection{Algebraic primitivity.}\label{SS:ap}  

\begin{proof}[\emph{\textbf{Proof of Theorem \ref{T:ap}}}]
By Corollary \ref{C:notap} it suffices to show that when $n$ and $m$ are distinct odd primes $\T(n,m)$ is not algebraically primitive, and to verify that $\T(n,m)$ is algebraically primitive for the $(n,m)$ in the theorem statement. This is straightforward from the above formulas for the degree of the invariant trace field and the formulas for genus in Corollary \ref{C:hoop}.
%
%
%
%
\end{proof}

\begin{rem}
The case when $n$ and $m$ are odd and relatively prime is discussed prior to Theorem 7.1 in \cite{BM}.
\end{rem}


\section{Lifting the pillowcase symmetry group}\label{S:Snm}

Let $n,m>1$ be natural numbers with $nm\geq 6$. Recall that $S(n,m)$ is the abelian square-tiled surface

$$M_{2nm}\left(\begin{array}{cccc}
nm-n-m & nm+n-m & nm+n+m & nm-n+m \\
nm+n-m & nm-n-m & nm-n+m & nm+n+m\\
\end{array}\right).$$

In this section we exhibit symmetries of $S(n,m)$ using the combinatorial model for square-tiled surfaces developed in \cite{W1}. An alternate approach, employed in \cite{BM}, exhibits the symmetries as lifts of M\"obius transformations. 

Recall that we denote by $\sigma_j$ the pillowcase symmetry that sends $z_1$ to $z_j$ (figure \ref{F:sigmas}). So, for example, the identity is $\sigma_1$. We will produce commuting lifts $\tilde{\sigma_2}$ and $\tilde{\sigma_4}$ of $\sigma_2$ and $\sigma_4$ to $S(n,m)$. 

\begin{proof}[\emph{\textbf{Proof of Proposition \ref{P:kleinlift}}}]
We begin by showing that any pair of lifts  $\tilde{\sigma_2},  \tilde{\sigma_4}$ which each have a fixed point automatically commute. The proof proceeds by finding formulas for these lifts. 

Let $\alpha_j$ be the free homotopy class of oriented loops about $z_j$, and let $T_j$ be the corresponding deck transformation. In terms of the combinatorial model for $M_N(A)$, 
$$T_j \left(\begin{array}{c} c_1\\c_2\end{array}\right)_{w/b}= \left(\begin{array}{c} c_1\\c_2\end{array}\right)_{w/b}+ \col_j(A).$$
By this we mean that the deck transformation $T_j$ sends each square to the square of the same color (white or black) with column label changed by addition of $\col_j(A)$ \cite[Section \ref{SS:comb}]{W1}.

The involution $\sigma_2$ transposes $\alpha_1$ and $\alpha_2$ as well as $\alpha_3$ and $\alpha_4$. It follows that if $\tilde{\sigma_2}$ is a lift of $\sigma_2$ with a fixed point, then the following equalities of deck transformations hold
 \begin{eqnarray}\label{E:s2intertwine}
 \tilde{\sigma_2} T_1 \tilde{\sigma_2} = T_2 \quad \text{and}\quad \tilde{\sigma_2} T_3 \tilde{\sigma_2} = T_4.
 \end{eqnarray}
 From these relations a more general commutation relation immediately follows for $S(n,m)$, since conjugation by $\tilde{\sigma_2}$ induces a automorphism of the deck group, and there is a unique such automorphism compatible with \ref{E:s2intertwine}. More concretely, applying the permutation $(12)(34)$ to the columns of $S(n,m)$ has the effect of transposing the entries in each column, so 
 \begin{eqnarray}\label{E:s2intertwine2}
 \tilde{\sigma_2} \circ\left(\begin{array}{c} c_1\\c_2\end{array}\right)\circ \tilde{\sigma_2} = \left(\begin{array}{c} c_2\\c_1\end{array}\right).
 \end{eqnarray}
Here $\left(\begin{array}{c} c_1\\c_2\end{array}\right)$ is in the column span, and we are identifying the column span with the deck group of $S(n,m)$.
 
Similarly, if $\tilde{\sigma_4}$ is any lift of $\sigma_4$ with a fixed point, 
 \begin{eqnarray} \label{E:s4intertwine}
 \tilde{\sigma_4} T_1 \tilde{\sigma_4} = T_4 \quad \text{and}\quad \tilde{\sigma_4} T_2 \tilde{\sigma_4} = T_3.
 \end{eqnarray}
  Again a more general commutation relation immediately follows: 
 \begin{eqnarray}\label{E:s4intertwine2}
 \tilde{\sigma_4}\circ \left(\begin{array}{c} c_1\\c_2\end{array}\right) \circ\tilde{\sigma_4} = \left(\begin{array}{c} -c_2\\-c_1\end{array}\right).
 \end{eqnarray}

 We are assuming that $\tilde{\sigma_2}$ has a fixed point. It must be on a 12 edge (a lift of the edge joining $z_1$ and $z_2$) or a 34 edge, and by acting by $\tilde{\sigma_4}$ if required we may assume it is on a 34 edge. (Commuting maps preserve each other's fixed point sets.)  Assume without loss of generality that $\tilde{\sigma_2}$ fixes the center of the 34 edge of the square $\left(\begin{array}{c} 0\\ 0\end{array}\right)_w$. So $$\sigma_2  \left(\begin{array}{c} 0\\ 0\end{array}\right)_{w/b}= \left(\begin{array}{c} 0\\ 0\end{array}\right)_{b/w}.$$
 Using the commutation relation \ref{E:s2intertwine2}, we then get  
 $$\sigma_2  \left(\begin{array}{c} c_1\\ c_2 \end{array}\right)_{w/b}= \left(\begin{array}{c} c_2\\ c_1 \end{array}\right)_{b/w}. $$
 
 Similarly, assume that $\tilde{\sigma_4}$ fixes the 14 edge of $\left(\begin{array}{c} d_1\\ d_2\end{array}\right)$. If either $n$ or $m$ is odd, an elementary computation gives that the column span (deck group) of $S(n,m)$ is $$\Span\left\{ \left(\begin{array}{c} -m \\ -m\end{array}\right),\left(\begin{array}{c} -n \\ n\end{array}\right)\right\}.$$
If both $n$ and $m$ are even, the column span is 
$$\Span\left\{ \left(\begin{array}{c} -2m \\ -2m\end{array}\right),\left(\begin{array}{c} -2n \\ 2n\end{array}\right),\left(\begin{array}{c} -n-m\\ n-m\end{array}\right)\right\}.$$ In the case that either $n$ or $m$ are odd, $\tilde{\sigma_2}$ commutes with the deck transformation $\left(\begin{array}{c} m \\ m\end{array}\right)$.
In the case that both $n$ and $m$ are even, $\left(\begin{array}{c} m \\ m\end{array}\right)$ is not a deck transformation, but nonetheless $\tilde{\sigma_2}$ commutes with the deck transformation $\left(\begin{array}{c} 2m \\ 2m\end{array}\right)$. 

Hence, applying a power of this deck transformation we may assume, if $n$ or $m$ is odd, that $d_1=-d_2$; if $n$ and $m$ are both even we may assume that either $d_1=-d_2$ or $d_1=-d_2-2m$. Since the 14 edge of $\left(\begin{array}{c} d_1\\ d_2\end{array}\right)_w$ is fixed by $\tilde{\sigma_4}$, we find that
\begin{eqnarray*}
\tilde{\sigma_4} \left(\begin{array}{c} d_1\\ d_2\end{array}\right)_w &=& \left(\begin{array}{c} d_1+nm-n+m\\ d_2+nm+n+m\end{array}\right)_b\\
\tilde{\sigma_2} \left(\begin{array}{c} d_1+nm-n+m\\ d_2+nm+n+m\end{array}\right)_b &=& \left(\begin{array}{c} d_1\\ d_2\end{array}\right)_w.
\end{eqnarray*}
Applying the commutation relations gives one of two possible formulas for $\tilde{\sigma_4}$. If $d_1=-d_2$, we conclude that 
\begin{eqnarray*}
\tilde{\sigma_4} \left(\begin{array}{c} c_1\\ c_2\end{array}\right)_w &=& \left(\begin{array}{c} -c_2 + nm-n+m\\ -c_1+nm+n+m\end{array}\right)_b \\
\tilde{\sigma_4} \left(\begin{array}{c} c_1\\ c_2\end{array}\right)_b &=& \left(\begin{array}{c} -c_2 + nm+n+m\\ -c_1+nm-n+m\end{array}\right)_w.
\end{eqnarray*}
If $d_1=-d_2-2m$, we conclude that
\begin{eqnarray*}
\tilde{\sigma_4} \left(\begin{array}{c} c_1\\ c_2\end{array}\right)_w &=& \left(\begin{array}{c} -c_2 + nm-n-m\\ -c_1+nm+n-m\end{array}\right)_b \\
\tilde{\sigma_4} \left(\begin{array}{c} c_1\\ c_2\end{array}\right)_b &=& \left(\begin{array}{c} -c_2 + nm+n-m\\ -c_1+nm-n-m\end{array}\right)_w.
\end{eqnarray*}
 
 A direct computation gives that using either choice of $\tilde{\sigma_4}$, the two involutions $\tilde{\sigma_2}$ and $\tilde{\sigma_4}$ commute and hence provide a lift of the pillowcase symmetry group to $S(n,m)$.

To see that $\tilde{\sigma_2}$  maps each horizontal cylinder to itself, one can check in the combinatorial model that $\tilde{\sigma_2}\left(\begin{array}{c} r_1\\ r_2\end{array}\right)_w$ and $\left(\begin{array}{c} r_1\\ r_2\end{array}\right)_b$ always differ by a power of the deck transformation $T_1\circ T_4$. The deck transformation $T_1\circ T_4$ is a ``horizontal translation" (figure \ref{F:OverTwo}), and all squares of the same same color in any given horizontal cylinder are related by a power of $T_1\circ T_4$.

\begin{figure}[h]
\includegraphics[scale=0.3]{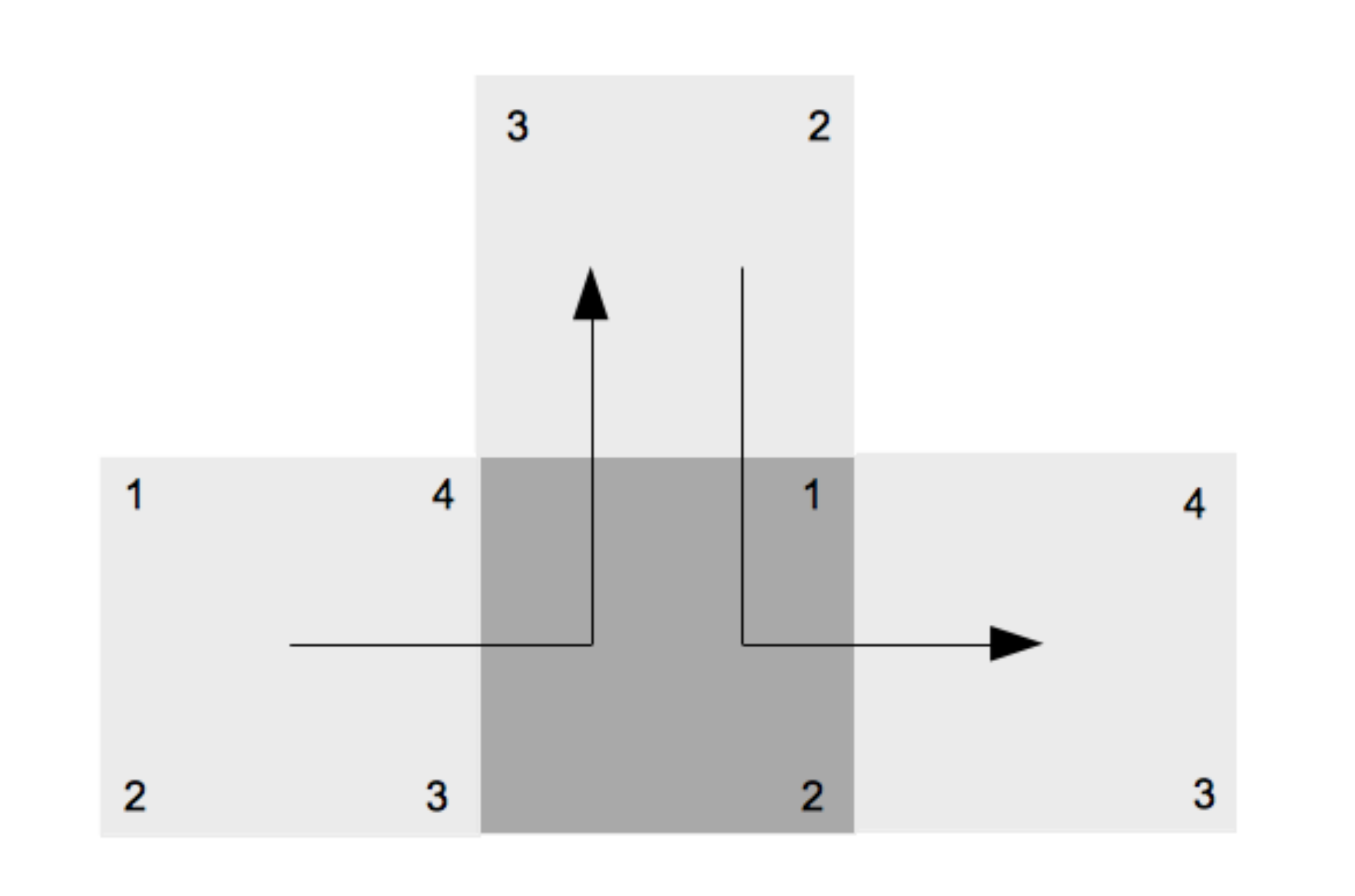}
\caption{The deck transformation $\alpha_1\circ\alpha_4$ translates squares horizontally.}
\label{F:OverTwo}
\end{figure}

The situation is similar for vertical cylinders and $\tilde{\sigma_4}$. 
\end{proof}

\begin{proof}[\emph{\textbf{Proof Lemma \ref{L:action}}}]
To begin, we comment on the claim in Proposition \ref{P:periods} that the nonzero bundles $H^1(\br)$ have $$\dim H^{1,0}(\br)=\dim H^{0,1}(\br)=1.$$ 
In general, the dimension of $H^1(\br)$ is $t(\br)+t(-\br)-2$ \cite[Proposition \ref{P:dim}]{W1}. The last two rows of $S(n,m)$ are obtained by negating the first two. Therefor if $\br$ has no zero entries, $$t_1(\br)+t_3(\br)=1=t_2(\br)+t_4(\br).$$
In this case $t(\br)=2=t(-\br)$ and hence $\dim H^{1,0}(\br)=\dim H^{0,1}(\br)=1.$ If $\br$ is nonzero but has a pair of zero entries, then $t(\br)=1=t(-\br)$ and $H^1(\br)$ is zero. So, henceforth we may assume $\br$ has no zero entries. 

The Klein group action follows directly from the commutation relations \ref{E:intertwine}, and the definition of $H^1(\br)$ in \cite{W1} as simultaneous eigenspaces of the deck group.

Finally, if $H^1(\br)$ is fixed by $\tilde{\sigma_3}$, then $r_1=r_3$ and $r_2=r_4$. However, because last two rows of $S(n,m)$ are obtained by negating the first two, $r_1=-r_3$ and $r_2=-r_4$. Since the $r_i$ are nonzero, $r_1=r_2=r_3=r_4=nm$. In particular, $H^1(\br)$ is also fixed by $\tilde{\sigma_2}$ and $\tilde{\sigma_4}$.

Suppose that $H^1(\br)$ is preserved by some involution $\tilde{\sigma}\in \{\tilde{\sigma_2}, \tilde{\sigma_4}\}$. Consider $\omega$ in $H^{1,0}(\br)$. Since $H^{1,0}(\br)$ has dimension 1, $\omega$ is either fixed or negated by the involution. By Proposition \ref{P:kleinlift}, $\tilde{\sigma}$ has a fixed point, which may not be a lift of one of the $z_j$ because the zeros of $\omega$ are at lifts of the $z_j$ (see the formula for $\omega$ in Part I). At the fixed point the involution acts by rotation by $\pi$. The coefficient of $dz$ in $\omega$ remains unchanged, and the $dz$ is negated. So $\omega$ is negated. The same argument shows that $\tilde{\sigma}$ negates $H^{0,1}(\br)$. 
\end{proof}


\section{Computing generators}\label{S:eqns}

In this section we compute the generators for $\T(n,m)$. We have chosen to record these computations here in full detail because they resolve \cite[Question 19]{H}.

\begin{proof}[\textbf{\emph{Proof of Theorem \ref{T:eqns}}}]
We consider a fiber $\S'_\infty$ of $\S'$ corresponding to $\lambda=\infty$. Letting $\lambda$ go to infinity has the effect of pinching all the vertical core curves of $S(n,m)$. Hence the noded Riemann surface $\S'_\infty$ consists of two components which are interchanged by $\tilde{\sigma_4}$, one, call it $P$, containing lifts of $z=0$ and $z=1$, and the other containing lifts of $z=\lambda$ and $z=\infty$. The function field of the first component $P$ is generated over $\bC(z)$ by $w_{1}$ and $w_{2}$, where $$w_{1}^N=z^{nm-n-m}(z-1)^{nm+n-m}\quad \text{and}\quad w_{2}^N=z^{nm+n-m}(z-1)^{nm-n-m}.$$

The desired fiber of the Bouw-M\"oller curve is given by the quotient of $P$ by $\tilde{\sigma_2}$. To determine this quotient, we calculate its function field $F$, which is the subset of the function field $\bC(z)[w_1, w_2]$ fixed by the action of $\tilde{\sigma_2}$. Note that $\tilde{\sigma_2}$ is given in general by a lift of $$z\mapsto \lambda \frac{z-1}{z-\lambda}=\frac{z-1}{\frac{z}{\lambda}-1},$$ hence the action of  $\tilde{\sigma_2}$ on the component $P$ of $\C'_\infty$ is a lift of the involution $z\mapsto 1-z$. All we will need to know about $\tilde{\sigma_2}$ is that when $\lambda\neq \infty$, it always has a fixed point, and hence, applying  $\tilde{\sigma_4}$ to the fixed point if necessary, it has a fixed point on the edge joining $z_1$ and $z_2$. Taking a limit, we find that when $\lambda=\infty$, $\tilde{\sigma_2}$ always has a fixed point on $P$ over $z=\frac12$.  

\emph{Case 1: $m$ odd.} Set $$p=w_1/w_2,\quad \text{so}\quad p^m=\frac{z-1}{z},$$ and set $$y=w_1 p^{\frac{m-1}2}/(z-1),\quad \text{so} \quad y^{2n}=z^{-1}(z-1)^{-1}.$$ Since $\tilde{\sigma_2}$ is a lift of the involution $z\mapsto 1-z$,  $\tilde{\sigma_2}(p)=\frac{\xi^l_m}{p}$ for some $l$, where $\xi_m$ is a primitive $m$-th root of unity. Since $m$ is odd, we can replace $p$ with $p$ times a $m$-th root of unity and assume that $\tilde{\sigma_2}(p)=\frac{1}{p}$.

Set $u=p+\frac1{p}$, so $u$ is fixed by $\tilde{\sigma_2}$. Observing that 
\begin{equation}\label{E:zp}
z=\frac1{1-p^m}, \quad z-1=\frac{p^m}{1-p^m},
\end{equation}
and 
\begin{equation}
\frac{(p-\xi_m^j)(p-\xi_m^{-j})}{p}=u-\xi_m^j-\xi_m^{-j}=u-2\cos\left(\frac{2\pi j}{m}\right),
\end{equation}
we conclude that 
\begin{equation}\label{E:yu1}
y^{2n}=\frac{(p^m-1)^2}{p^m}=(u-2) \prod_{j=1}^{\frac{m-1}2} \left(u-2\cos\frac{2\pi j}{m}\right)^2.
\end{equation}
It follows that $y$ is either fixed or negated by $\tilde{\sigma_2}$. But if $y$ were negated, there would be no fixed point above $z=\frac12$. So it must be that $y$ is fixed. 

We now claim that $\bC(y,u)=F$. Clearly, $\bC(y,u)\subset F$. Since $$p^2-up+1=0,$$ we have that $\bC(z)[w_1, w_2]=\bC(y,u)[p]$ is a degree two extension of $\bC(y,u)$. Since the map from $P$ to $P/\tilde{\sigma_2}$ is degree two, it follows that $\bC(z)[w_1, w_2]$ is a degree two extension of $F$, so the claim is proved. We have computed the function field of $P/\tilde{\sigma_2}$, and the given equation follows immediately in light of equation \ref{E:yu1}. 

When $\lambda\neq 0$, the abelian differential differential on $P$ is 
\begin{equation}\label{E:om1}
\omega+\tilde{\sigma_2}(\omega)+\tilde{\sigma_4}(\omega+\tilde{\sigma_2}(\omega)).
\end{equation}
We may take $\omega=z^{-s_1}(z-1)^{-s_2}(z-\lambda)^{-s_3}dz$, where 
\begin{eqnarray*}
&& s_1=\frac{nm-n-m}{2mn}, \quad s_2=\frac{nm+n-m}{2mn}, 
\\&&
s_3=\frac{nm+n+m}{2mn},  \quad s_4=\frac{nm-n+m}{2mn}.
\end{eqnarray*}
Using the symbol $\propto$ to denote equality up to a nonzero scalar, we can compute that, 
$$\tilde{\sigma_4}(\omega) \propto z^{-s_4} (z-1)^{-s_3}(z-\lambda)^{-s_2} dz.$$ The power of $z-\lambda$ in $\tilde{\sigma_4}(\omega)$ is greater than that in $\omega$. When $\lambda\to\infty$  there is a renormalization, with the effect that only the final term of \ref{E:om1} survives. Hence the abelian differential on $P$ is given by $$z^{-s_4} (z-1)^{-s_3}dz+\tilde{\sigma_2}(z^{-s_4} (z-1)^{-s_3}dz).$$ 

Now, using 
\begin{equation}\label{E:dz}
dz\propto \frac{p^{m-1}}{(1-p^m)^2}dp
\end{equation}
we compute that, 
$$z^{-s_4} (z-1)^{-s_3}dz \propto \frac{y p^{\frac{m-3}2} dp}{1-p^m}= \frac{y}{(u-2)\prod_{j=1}^{\frac{m-1}2} \left(u-2\cos \frac{2\pi j}{m}\right)} \frac{p-1}{p^2}dp.$$

Since 
\begin{eqnarray*}
\frac{p-1}{p^2}dp+\tilde{\sigma_2}\left(\frac{p-1}{p^2}dp\right) &=& \frac{p-1}{p^2}dp+ \frac{p^{-1}-1}{p^{-2}}\frac{-dp}{p^2}
\\&=&\left(1-\frac1{p^2}\right)dp=du.
\end{eqnarray*}
the formula for the differential follows. 

\emph{Case 2: $m$ even.} Set $$p=w_1/w_2,\quad \text{so} \quad p^m=\frac{z-1}{z},$$ and set $$q=w_1 p^{\frac{m}2}/(z-1),\quad \text{so} \quad q^{2nm}=z^{-n-m}(z-1)^{n-m}.$$ 

Over $z=\frac12$, we have $p^m=-1$. Note $\tilde{\sigma_2}(p)=\frac{\xi_{m}^l}{p}$ for some $l$. The existence of a fixed point for $\tilde{\sigma_2}$ over $z=\frac12$ gives that $l$ is odd, and by replacing $p$ with $p$ times an $m$-th root of unity we may assume that $\tilde{\sigma_2}(p)=\frac{\xi_{m}}{p}$.

Set $u=\xi_{2m}^{-1}p+\frac{\xi_{2m}}{p}$, and notice that 
\begin{eqnarray}\label{E:xi2}
u-\xi_{2m}^{2i-1}-\xi_{2m}^{1-2i}&=&\xi_{2m}^{-1}\frac{(p-\xi_m^i)(p-\xi_m^{1-i})}{p}
\\
 u-2&=&\xi_{2m}^{-1}\frac{(p-\xi_{2m})^2}{p}.
\end{eqnarray}

Direct computation shows that $$q^{2n}\propto [z(z-1)]^{-1} \left[\frac{z-1}{z} \right]^{\frac{n}{m}} \propto \frac{(p^m-1)^2}{p^m}p^n.$$ If we set $y=\frac{q(p-\xi_{2m})}{p}$, we get 
$$y^{2n}\propto \frac{(p^m-1)^2}{p^m} \frac{(p-\xi_{2m})^{2n}}{p^n}.$$
Since $p$ is an $m$-th root of unity when $z=\frac12$, it follows that $y\neq 0$ when $z=\frac12$. Now, using \ref{E:xi2} and possibly modifying $y$ by a root of unity, we get 
\begin{equation}\label{E:curve2}
y^{2n}= (u-2)^{n}\prod_{j=1}^{\frac{m}2} \left(u-2\cos\frac{\pi (2 j-1)}{m}\right)^2.
\end{equation}

It follows from this description that $\tilde{\sigma_2}$ fixes or negates $y$; again the existence of a fixed point requires that $y$ is fixed and not negated. As in the first case we see that $\bC(u,y)$ is the function field of $P/\langle\tilde{\sigma_2}\rangle$. The only difference for the case when $n$ is even is that the polynomial \ref{E:curve2} becomes reducible.

Now we wish to express the differential $$z^{-s_4} (z-1)^{-s_3}dz+\tilde{\sigma_2}(z^{-s_4} (z-1)^{-s_3}dz)$$ in terms if $u$ and $y$. 
\begin{eqnarray*}
z^{-s_4} (z-1)^{-s_3}dz
& \propto&\frac{qp^{\frac{m}2-1}}{z-1}dz 
\\& \propto&
q \frac{p^{\frac{m}2-2}}{1-p^m}  dp
\\& \propto&
y \frac{p}{(p-\xi_{2m})^2}\frac{p^{\frac{m}2}}{1-p^m} \frac{(p-\xi_{2m})dp}{p^2}
\\& \propto&
\frac{y}{(u-2)\prod_{j=1}^{\frac{m}2} \left(u-2\cos \frac{\pi (2j-1)}{m}\right)} \frac{(p-\xi_{2m})dp}{p^2}.
\end{eqnarray*}
Now, since 
\begin{eqnarray*}
\frac{(p-\xi_{2m})dp}{p^2}+\tilde{\sigma_2}\left(\frac{(p-\xi_{2m})dp}{p^2}\right)&=&
du,
\end{eqnarray*}
the formula for the differential follows. 
\end{proof}



\bibliography{mybib}{}
\bibliographystyle{amsalpha}
\end{document}